\documentclass[11pt]{amsart}

\usepackage{amsmath, amssymb, amsfonts}
\usepackage{graphicx, overpic}
\usepackage{color}
\usepackage{enumerate}
\usepackage{multicol, verbatim}

\usepackage[margin=1.25in]{geometry}

\numberwithin{equation}{section}
\numberwithin{figure}{section}

\theoremstyle{plain}
\newtheorem{thm}{Theorem}[section]

\newtheorem*{Mainthm}{Main Theorem}
\newtheorem{lemma}[thm]{Lemma}
\newtheorem{prop}[thm]{Proposition}
\newtheorem{cor}[thm]{Corollary}
\theoremstyle{definition}
\newtheorem{defn}[thm]{Definition}

\newtheorem{remark}[thm]{Remark}

\newtheorem{notation}[thm]{Notation}

\newtheorem{construction}[thm]{Construction}

\newcommand{\Z}{\mathbb{Z}}
\newcommand{\R}{\mathbb{R}}
\newcommand{\N}{\mathbb{N}}

\renewcommand{\H}{\mathbb{H}}

\newcommand{\cB}{\mathcal{B}}

\newcommand{\cE}{\mathcal{E}}

\newcommand{\cL}{\mathcal{L}}

\newcommand{\cQ}{\mathcal{Q}}
\newcommand{\cR}{\mathcal{R}}

\newcommand{\G}{\Gamma}

\newcommand{\CAT}{\operatorname{CAT}}
\newcommand{\Aut}{\operatorname{Aut}}
\newcommand{\Out}{\operatorname{Out}}
\newcommand{\dom}{\operatorname{dom}}
\newcommand{\Int}{\operatorname{Int}}
\newcommand{\vol}{\operatorname{vol}}
\newcommand{\la}{\langle}
\newcommand{\ra}{\rangle}
\newcommand{\p}{\partial}

\newcommand{\doublebndry}{\partial^2F_n}

\newcommand{\hi}{\hat{\iota}}   
\newcommand{\pT}{\partial T}
\newcommand{\oT}{\overline{T}}
\newcommand{\hT}{\widehat{T}}

\newcommand{\out}{\textup{Out}(F_n)}
\newcommand{\aut}{\textup{Aut}(F_n)}
\newcommand{\Lam}{\Lambda}
\newcommand{\sig}{\sigma}
\newcommand{\from}{\colon}
\newcommand{\ga}{\alpha}
\newcommand{\wt}{\widetilde}
\newcommand{\pf}{\partial f_0}
\newcommand{\Wh}{\operatorname{Wh}}
\newcommand{\diam}{\operatorname{diam}}

\newcommand{\scarpet}{Sierpinski }

\definecolor{amethyst}{rgb}{0.6, 0.4, 0.8}

\newcommand{\hide}[1]{}

\title{The visual boundary of hyperbolic free-by-cyclic groups}
\author{Yael Algom-Kfir}
\author{Arnaud Hilion}
\author{Emily Stark}
\date{\today}
\thanks{The first and third authors were supported by ISF grant 1941/14. The second author was supported by the grant ANR-16-CE40-0006 (ANR Dagger). The third author was supported by the Azrieli Foundation and Zuckerman Foundation. }

\begin{document}
 
\maketitle

\begin{abstract} 
Let $\phi$ be an atoroidal outer automorphism of the free group $F_n$. We study the Gromov boundary of the hyperbolic group $G_{\phi} = F_n \rtimes_{\phi} \Z$. We explicitly describe a family of embeddings of the complete bipartite graph $K_{3,3}$ into $\partial G_\phi$. 
To do so, we define {\it the directional Whitehead graph} and prove that an indecomposable $F_n$-tree is {\it Levitt type} if and only if one of its directional Whitehead graphs contains more than one edge. 
As an application, we obtain a direct proof of Kapovich-Kleiner's theorem \cite{kapovichkleiner} that $\partial G_\phi$ is homeomorphic to the Menger curve if the automorphism is atoroidal and fully irreducible. 
\end{abstract}


\vskip.23in

\section{Introduction}

  The family of free-by-cyclic groups is an intriguing example of a collection of groups that is simply defined, yet has a rich geometric structure theory.
  Denote by $F_n = \la x_1, \ldots, x_n \ra$ the free group on $n \geq 2$ generators. Let $\aut$ be the group of automorphisms of $F_n$, and let $\out$ be the quotient of $\aut$ by the normal subgroup of inner automorphisms.
  Given an outer automorphism $\phi \in \out$, the group $G_{\phi}$ is defined by the HNN presentation 
 \begin{equation}\label{hnnpresentation}
 G_{\phi} := F_n \rtimes_{\phi} \Z = \la \, x_1, \ldots, x_n, t \,\, | \,\, t^{-1}x_it=\Phi(x_i), \, 1 \leq i \leq n \, \ra \end{equation}
 where $\Phi\in\aut$ is any automorphism of $F_n$ in the outer class $\phi$. (Different choices of $\Phi$ yield Tietze-equivalent presentations of $G_{\phi}$.)
  There is a satisfying correspondence between geometric properties of $G_\phi$ and algebraic properties of $\phi$. Brinkmann \cite{brinkmann00} proved that $G_\phi$ is hyperbolic precisely when $\phi$ is {\it atoroidal}, which means that no conjugacy class of $F_n$ is invariant under a power of $\phi$. 
  In the case that $G_\phi$ is Gromov hyperbolic, a theorem of Bowditch \cite[Theorem 6.2]{bowditch} implies that the visual boundary of $G_\phi$ contains a local cut point if and only if $G_\phi$ splits over a virtually infinite cyclic subgroup. Kapovich--Kleiner \cite[proof of Corollary 15]{kapovichkleiner} proved that in this case some power of $\phi$ preserves a free splitting of $F_n$. Thus, there is a trichotomy: $\partial G_\phi$ contains a local cut point if and only if $G_\phi$ splits over a virtually infinite cyclic subgroup if and only if a power of $\phi$ preserves a free splitting of $F_n$. 

  If $G_{\phi}$ is a hyperbolic free-by-cyclic group, then the cohomological dimension of $G_{\phi}$ is two, hence the boundary of $G_{\phi}$ is $1$-dimensional by work of Bestvina--Mess \cite{bestvinamess}. Kapovich--Kleiner \cite[Theorem 4]{kapovichkleiner} proved that if $G$ is a hyperbolic group with one-dimensional boundary, and $G$ does not split over a finite or virtually cyclic subgroup, then the boundary of $G$ is homeomorphic to either the unit circle $S^1$, the \scarpet carpet, or the Menger curve. A deep theorem of Tukia--Gabai--Casson--Jungreis \cite{tukia,gabai,cassonjungreis} classifies the hyperbolic groups with boundary homeomorphic to $S^1$ as precisely the groups that act discretely and cocompactly by isometries on the hyperbolic plane. Kapovich--Kleiner \cite[Theorem 5]{kapovichkleiner} characterize the structure of hyperbolic groups that have boundary homeomorphic to the \scarpet carpet, and it follows from their characterization that any such group has negative Euler characteristic (see Section~\ref{sec_boundary} of the present article). Consequently, a hyperbolic free-by-cyclic group cannot have boundary homeomorphic to $S^1$ or the \scarpet carpet, since a free-by-cyclic group has Euler characteristic equal to zero. Thus, the boundary of a hyperbolic free-by-cyclic group which does not split over a virtually cyclic group must be homeomorphic to the Menger curve.

  The starting point of this project was to explicitly find a non-planar set inside $\partial G_\phi$ when $G_{\phi}$ is hyperbolic. Kuratowski  \cite{kuratowski} proved that a graph is non-planar if and only if it admits an embedding of the complete bipartite graph $K_{3,3}$ or the complete graph $K_5$. 

  \begin{Mainthm}
  If $\phi$ is an atoroidal automorphism in $\out$, then $\partial G_\phi$ contains a copy of the complete bipartite graph $K_{3,3}$. 
  \end{Mainthm}

  In order to find this non-planar set,  we use the following facts and constructions. If $\phi$ is fully irreducible, then $\phi$ acts on the set of projective classes of very small minimal $F_n$-trees by north-south dynamics \cite{levittlustig}. Denote the attractor by $T_+$ and the repeller by $T_-$. Levitt and Lustig \cite{levittlustig} defined for any free $F_n$ tree $T$ with dense orbits a map $\cQ \from \partial F_n \to \widehat T$ where $\widehat T$ is the compactification of the metric completion of $T$. Of particular interest to us are the maps $\cQ_-$ and $\cQ_+$ associated to the trees $T_-$ and $T_+$. Mitra (Mj) \cite{mitra} proved there exists a continuous surjection $\hi \from \partial F_n \to \partial G_\phi$, that extends the injection $\iota \from F_n \to G_{\phi}$ in the presentation \eqref{hnnpresentation}. The map $\hi$ is called the {\it Cannon-Thurston map for the subgroup $F_n \leq G_\phi$}. In \cite{kapovichlustig}, Kapovich and Lustig proved that preimages of $\hi$ are finite and uniformly bounded. Moreover, they showed that $\hi$ factors through both $\cQ_+$ and $\cQ_-$. That is, there exist continuous surjections $\cR_\pm \from \widehat T_\pm \to \partial G_\phi$ so that $\hi = \cR_+ \circ \cQ_+ = \cR_- \circ \cQ_-$. The maps $\cR_\pm$ are both onto, and restricted to $T_\pm$ they are injective and their images are disjoint.  

  To prove the Main Theorem, we realize the embedded $K_{3,3}$ in $\p G_{\phi}$ as a union of a part that lies in the tree $T_+$ and a part that lies in the tree $T_-$. If $\phi$ admits the existence of these subsets in a particular way, then we say that $\phi$ {\it satisfies the $T_\pm$-pattern}; see Definition~\ref{def_Tpic}. Given a tree $T$ with dense orbits, we introduce the {\it directional Whitehead graph} (in Definition~\ref{def_dirWH}) to capture certain asymptotic relations between leaves of the lamination dual to $T$. We prove in Lemma~\ref{DWvsTLevitt} that for any indecomposable tree $T$, there is the following correspondence.
  
    \begin{center}
    \begin{tabular}{p{6 cm} c p{7cm}}
	A directional 
	Whitehead graph of~$T$ contains
	more than one edge.  & $\iff$ &
	$T$ has {\it Levitt type}; in particular, there exists a system of partial isometries associated to $T$ for which the Rips Machine never halts. See Definition~\ref{def_levitt_type}.
    \end{tabular} 
    \end{center}

    \noindent  In Lemma \ref{paternIffDw} we prove
    \begin{center}
    \begin{tabular}{p{6 cm} c p{6 cm}}
	A directional Whitehead graph of $T_-$ contains more than one edge & $\iff$ &
	$\phi$ satisfies the $T_\pm$-pattern.
    \end{tabular} 
    \end{center} 
   The $T_\pm$-pattern maps into $\partial G_\phi$ to form an embedded $K_{3,3}$, as shown in Proposition~\ref{patternImpliesK33}. In Section~\ref{sec_boundary}, we explain how the Main Theorem gives a direct proof of a result of Kapovich--Kleiner \cite[Section 6]{kapovichkleiner} which states that, if $\phi$ is fully irreducible, then $\partial G_\phi$ is homeomorphic to the Menger curve. 

  There is a striking and well-developed analogy between fully irreducible automorphisms of free groups and pseudo-Anosov homeomorphisms of surfaces. Let $\Sigma$ be a surface with boundary. A pseudo-Anosov homeomorphism of $\Sigma$ that fixes the boundary components of $\Sigma$ gives rise to a fully irreducible automorphism $\psi$ of $\pi_1(\Sigma)$, a non-abelian free group. In this case, the automorphism $\psi$ is not atoroidal. While the associated free-by-cyclic group $G_{\psi}$ is not a hyperbolic group, the group is $\CAT(0)$ with isolated flats. Hence, the group has a unique $\CAT(0)$ boundary by work of Hruska--Kleiner \cite{hruskakleiner}.  Ruane \cite{ruane} proved that the $\CAT(0)$ boundary of $G_{\psi}$ in this case is homeomorphic to the Sierpinski carpet; in particular, the $\CAT(0)$ boundary of $G_{\psi}$ is planar. A key distinction between these settings is the difference in the systems of partial isometries representing the attracting and repelling trees associated to the automorphisms. The process of iterating the Rips Machine on the system associated to the pseudo-Anosov will halt, producing an interval exchange transformation; while on a system associated to a fully irreducible atoroidal automorphism, the Rips Machine will  not halt, and will produce a band system with three overlapping bands (Proposition \ref{prop:levitt implies overlap}). Our methods illustrate how this difference in dynamics leads to the difference in the topology of the boundary: one is planar, while the other is not.
    
  The quasi-conformal homeomorphism type of the visual boundary of a hyperbolic group is a complete quasi-isometry invariant \cite{p96}. The quasi-isometry classification of hyperbolic free-by-cyclic groups is open. Our hope is that the new-found understanding of these structures in the boundary will allow us to understand more about the quasi-conformal structure on the boundary of these groups as well, but, this is beyond the scope of the current paper.
  
  In the process of proving our main theorem, we establish new results which may be of independent interest to specialists in the field. For instance, we further investigate systems of partial isometries on compact trees in continuation of the work of Coulbois, Hilion and Lustig  \cite{coulboishilionlustig,coulboishilion,coulboishilion_bot}. Precise definitions of the objects considered here can be a bit technical and will be given in Sections~\ref{sec:prelims}~and~\ref{subsec:rips_class}. In particular, combining our results with previous results of \cite{coulboishilion_bot}, we obtain the following characterization of the type (Levitt or surface) of a free mixing $F_n$-tree.

  \begin{thm}[Corollary~\ref{cor:type}]
  Let $T$ be a free mixing $F_n$-tree, and let $A$ be a basis of $F_n$.
  Let $S_0 = (K_0, A_0)$ be the associated system of partial isometries, and let $S_i= (K_i, A_i)$ denote the output after the $i^{th}$ iteration of the Rips Machine. Then
  \begin{itemize}
  \item $T$ is pseudo-surface if and only if $\vol\left(K_i^{\geq 3}\right)=0$ for some $i\in\N$;
  \item $T$ is Levitt type if and only if $\vol\left(K_i^{\geq 3}\right)>0$ for all $i\in\N$,
  \end{itemize}
  where $K_i^{\geq 3}$ is the subset of $K_i$ where at least 3 distinct partial isometries in $A_i$ are defined, and $\vol\left(K_i^{\geq 3}\right)$ denotes the volume of $K_i^{\geq 3}$.
  \end{thm}

  Another noteworthy result relates two possible definitions of the {\em ideal Whitehead graph} of a fully irreducible outer automorphism $\phi\in\out$. Handel--Mosher \cite[Chapter~3]{handelMosher_axes} introduced the ideal Whitehead graph to study the asymptotic relations between singular leaves of the attracting lamination of $\phi$; see Section~\ref{subsec_lams}. They define the graph $\Wh_{\phi}$ as the $F_n$-quotient of the graph whose vertex set is the union of nonrepelling fixed points in $\p F_n$ of principal automorphisms representing $\phi$; see Definition~\ref{def_principal}. An edge of $\Wh_{\phi}$ corresponds to a leaf of the attracting lamination $\Lam_+^\phi$ of $\phi$. For the second definition, let the graph $\Wh_{\Lam_+^\phi}$ be the $F_n$-quotient of the graph whose vertex set is the union of endpoints of singular leaves of $\Lam_+^\phi$, where a singular leaf has an asymptotic class containing more than one element. An edge of $\Wh_{\Lam_+^\phi}$ corresponds to a singular leaf of $\Lam_+^\phi$. We prove the following in Section~\ref{sec_wh_graphs_iso}.
  
  \begin{thm}\label{thm:Whitehead graph}
  The ideal Whitehead graphs $\Wh_{\phi}$ and $\Wh_{\Lam_+^\phi}$ of a fully irreducible outer automorphism $\phi\in\out$ are isomorphic.
  \end{thm}

 \subsection*{Acknowledgements} 
    We wish to thank the Mathematical Science and Research Institute for their hospitality in the fall of 2016.   

\section{Preliminaries} \label{sec:prelims}

  \subsection{Trees, directions, and the observers' topology} \label{subsec:trees}

     A metric space $(T,d)$ is an {\it $\R$-tree} if for any two points $x,y\in T$, there is a unique topological arc $p_{x,y}:[0,1]\rightarrow T$ connecting $x$ to $y$ and so that the image of $p_{x,y}$ is isometric to the segment $[0,d(x,y)]$. We denote the image of $p_{x,y}$ by $[x,y]$, and we refer to this arc as the {\it segment} in $T$ from $x$ to $y$. An arc is {\it non trivial} if it contains at least 2 distinct points.
     
     If $x \in T$ is a point, a {\it direction} at $x$ in $T$ is a component of $T \smallsetminus \{x\}$. A point $x \in T$ is a {\it branch point} if there are at least 3 directions at $x$; it is an {\it extremal point} if there is only one direction at $x$.
     Let $\Int(T)$ denote the {\it interior} of a tree $T$, i.e. the set of points of $T$ which are not extremal.
     Two arcs $[x,y]$, $[x,y']$ are {\it germ-equivalent} if $[x,y] \cap [x,y'] \neq\{x\}$; an equivalence class is a {\it germ} at $x$. The map that associates to the arc $[x,y]$ the direction containing $y$ induces a bijection from the set of germs at $x$ to the set of directions at $x$.
     
     A {\it ray} in $T$ is the image of an immersion $\R_+\rightarrow T$. Two rays $\rho$ and $\rho'$ are {\it asymptotic} if $\rho\cap\rho'$ has infinite diameter.
     Let $\p T$ denote the {\it Gromov boundary} of $T$; as a set, $\p T$ consists of asymptotic classes of rays in $T$.
     Let $\overline{T}$ denote the metric completion of $T$. The points in $\overline{T}\smallsetminus T$ are extremal in $\overline{T}$.
     Let $\hT = \overline{T} \cup \p T$.
     
     As explained by Coulbois--Hilion--Lustig \cite{coulboishilionlustig_ot}, the set $\hT$ may be equipped with the {\it observers' topology}, which makes the space homeomorphic to a dendrite.
     The observers' topology on $\hT$ is the topology generated (in the sense of a subbasis) by the set of directions in $\hT$, where the notion of direction in $T$ extends to $\hT$, and points of $\p T$ are extremal in $\hT$. This topology is weaker than the topology induced by the metric on an $\R$-tree. On any interval of $\overline{T}$, the metric topology and the observers' topology agree.

  \subsection{Systems of partial isometries and the band complex} \label{subsec:systems}

    \begin{defn} \label{def:partial_isom}
     A {\it compact forest} is a finite union of compact $\R$-trees. A {\it partial isometry} of a compact forest $K$ is an isometry $a:J \rightarrow J'$, where $J$ and $J'$ are compact subtrees of~$K$. The {\it domain} of $a$, denoted $\dom(a)$, is $J$, and the {\it range} of $a$ is $J'$. A partial isometry is {\it non-empty} if its domain is non-empty. A {\it system of partial isometries} is a pair $S = (K,A)$, where $K$ is a compact forest and $A$ is a finite collection of non-empty partial isometries of~$K$. 
    \end{defn}
 
    \begin{defn}
     Let $S = (K,A)$ be a system of partial isometries, and let $I = [0,1]$ denote the unit interval. For each $a_i \in A$, let $b_i \subset K$ be the domain of $a_i$. Let $B_i:=b_i \times I$ be called a {\it band}. The {\it band complex} $\cB$ is the quotient of $K \sqcup_i B_i$, where $b_i \times \{0\}$ is identified to the domain of $a_i$, and $b_i \times \{1\}$ is identified to the range of $a_i$ by isometries. 
    \end{defn}

   \begin{defn} \label{def_systems}
    Each band $B = b \times I \subset \cB$ is foliated by leaves of the form $x \times I$ for $x \in b$. The foliation of the bands yields a foliation of $\cB$.    
    A finite, infinite, or bi-infinite path $\gamma$ in $\cB$ is an {\it admissible leaf path} if $\gamma$ is a locally isometric, immersed path contained in a leaf.
    A {\it half-leaf based at $x \in K$} is an admissible leaf path $\rho:[0,\infty)\rightarrow \cB$ with $\rho(0)=x$. 
    A finite admissible path $\gamma$ travels through a finite sequence of bands $B_1, \dots, B_k$ and, in turn, corresponds to the sequence of partial isometries $a_1, \dots, a_k$ that give rise to the bands. The {\it domain} of $\gamma$ is then $\dom(a_k \circ \dots \circ a_1)$. Note that the domain contains $\gamma(0)$. This definition extends to admissible rays and bi-infinite lines. The {\it limit set} $\Omega$ of $S = (K,A)$ is the set of elements of $K$ which are in the domain of a bi-infinite admissible reduced path. The {\it lamination} $L(\cB)$ is the set of bi-infinite admissible leaf paths in $\cB$. For an $\R$-tree $J$, let $\mu_J$ denote the Lebesgue measure on $J$, which consists of the Lebesgue measures on the segments of $J$. The foliated space $\cB$ has a {\it transverse measure} associated to the Lebesgue measure on the components of $K$. 
     \end{defn}
       
       \begin{notation} \label{notation_leaf_dot}
        Let $x \in K$, and let $\ell = \ldots z_{-2}z_{-1}z_0z_1z_2 \ldots \subset L(\cB)$ be a bi-infinite admissible leaf path through $x$; so, $z_i \in A \cup A^{-1}$ for $i \in \Z$. Suppose $\ell_1 = z_0z_1z_2 \ldots$ and $\ell_2 = z_{-1}^{-1}z_{-2}^{-1} \dots$ denote the half-leaves based at $x$. Then, $\ell = \ell_1 \cup \ell_2$, and to record the point $x$ in the leaf $\ell$, we use the notation $\ell = \ldots z_{-2}z_{-1}.z_0z_1z_2 \ldots$.
       \end{notation}

   \begin{defn}
    Let $S = (K,A)$ be a system of partial isometries, and let $\cB$ be the associated band complex. Let $\widetilde{\cB}$ denote the universal cover of $\cB$. The foliation and transverse measure on $\cB$ lift to a foliation and transverse measure on $\widetilde{\cB}$. Collapsing each leaf in $\widetilde{\cB}$ to a point yields an $\R$-tree $T_{\cB}$, called the {\it dual tree} to $\cB$. 
   \end{defn}

  \subsection{The $\cQ$-map and dual lamination} \label{subsec:Q_laminations}
  
      \begin{defn}
      An {\it $F_n$-tree} is an $\R$-tree $T$ together with a homomorphism $\rho \from F_n \to \text{Isom}(T)$; the homomorphism is often repressed. 
      The action of $F_n$ on $T$ is {\it minimal} if $F_n$ does not leave invariant any non-trivial subtree of $T$. The action is {\it very small} if (1) all edge stabilizers are cyclic (i.e. $\{1\}$ or $\Z$); (2) for every non-trivial $g \in F_n$, the fixed subtree $Fix(g)$ is isometric to a subset of $\R$; and (3) $Fix(g)$ is equal to $Fix(g^p)$ for all $p \geq 2$. 
      \end{defn}

    \begin{thm} \cite{levittlustig} \cite[Proposition 2.3]{coulboishilionlustig_ot} \label{thm:Q_map}
      Let $T$ be a  minimal, very small, $F_n$-tree with dense orbits. There exists an $F_n$-equivariant surjective map \[ \cQ: \p F_n \rightarrow \hT\] which is continuous with respect to the observers' topology. In addition, points in $\pT$ have exactly one pre-image by $\cQ$.
       \qed
    \end{thm}   
    
  The $\cQ$-map given in the previous theorem may be used to define a lamination of $F_n$ as follows. 
  
    \begin{defn}
     The {\it double boundary} of $F_n$ is $\p^2F_n := (\p F_n \times \p F_n) \smallsetminus \Delta$, where $\Delta$ is the diagonal. Let $i:\p^2F_n \rightarrow \p^2F_n$ denote the involution that exchanges the factors. 
     The double boundary $\p^2F_n$ is endowed with the topology induced by the product topology, and $F_n$ acts diagonally on $\p^2F_n$.
     A {\it lamination of $F_n$} is a non-empty, closed, $F_n$-invariant, $i$-invariant subset of $\p^2F_n$. 
    \end{defn}

    \begin{defn} \label{def:lam}
      Let $F_n$ act by isometries on an $\R$-tree~$T$ so that the action is minimal, very small, and has dense orbits.
      Let $\cQ:\p F_n \rightarrow \hT$ be the $\cQ$-map given in Theorem \ref{thm:Q_map}. The {\it dual lamination of $T$} is 
      \[ L(T) = \{(X,Y) \in \p^2 F_n \,| \, \cQ(X) = \cQ(Y)\} .\] A {\it leaf} of the dual lamination $L(T)$ is a pair $(X,Y) \in L(T)$.         
    \end{defn}
 
 \subsection{Attracting and repelling trees and laminations of an outer automorphism} \label{subsec_lams}

    \begin{defn} (Outer space.)
      For $n \geq 2$, {\it Culler--Vogtmann outer space}, denoted $CV_n$, is the projectivized space of minimal, free, discrete actions by isometries of the free group $F_n$ on $\R$-trees. 
      Its topology is induced by embedding $CV_n$ into the space of length functions \cite{cullermorgan}.
      Let $\overline{CV}_n$ denote the compactification of $CV_n$, which is the set of projective classes of minimal, very small, $F_n$-trees \cite{cullermorgan,CohenLustig, OuterLimits}. Let 
       $\p CV_n = \overline {CV_n} \smallsetminus CV_n$. 
      These spaces admit an action of $\out$; an element $\phi \in \out$ sends an $F_n$-tree $(T,\rho)$ to $(T, \rho \circ \Phi)$, where $\Phi \in \Aut(F_n)$ is in the class $\phi$.  
      (For background, see \cite{cullervogtmann}, \cite{vogtmann}.)
    \end{defn}

    \begin{defn}
      An outer automorphism $\phi\in\out$ is {\it fully irreducible} (or, {\it iwip}) if no conjugacy class of a proper free factor of $F_n$ is fixed by a positive power of $\phi$.
    \end{defn}

   \begin{thm} \cite{levittlustig}
      If $\phi \in \Out(F_n)$ is fully irreducible, then $\phi$ acts on $\overline{CV}_n$ with North-South dynamics and projectively fixes two trees $T_+, T_- \in \p CV_n$. 
       \qed
    \end{thm}
  
   \begin{defn}
    We refer to the tree $T_+=T^{\phi}_+$ as the {\it attracting tree of $\phi$} and to the tree $T_- = T^{\phi}_-$ as the {\it repelling tree of $\phi$}. We omit the $\phi$-notation when the outer automorphism is clear from context.  Notice that $T^{\phi}_+ = T^{\phi^{-1}}_-$ and $T^\phi_- =T^{\phi^{-1}}_+$. 
   \end{defn}   
   
   If $\phi$ is a fully irreducible outer automorphism, then the group $F_n$ acts on the trees $T_+$ and $T_-$ with dense orbits. In fact, these actions exhibit a much stronger dynamical property. 
  
    \begin{defn}
     The action of $F_n$ on an $\R$-tree $T$ is {\it  mixing} if for any non-degenerate segments $I$ and $J$ in $T$ the segment $I$ is covered by finitely many translates of $J$: there exist finitely many elements $u_1, \ldots, u_k \in F_n$ so that $I \subset u_1J\cup u_2J \cup \ldots \cup u_kJ$. The action is {\it indecomposable} if, in addition, the elements $u_1, \ldots, u_k \in F_n$ may be chosen so that $u_iJ \cap u_{i+1}J$ is a non-degenerate segment for any $i = 1, \ldots, k-1$.
    \end{defn}

    \begin{lemma} \cite[Theorem 2.1]{coulboishilion_bot} \label{lemma_iwipindecomposable} 
         The actions of $F_n$ on the attracting and repelling trees of a fully irreducible outer automorphism of $F_n$ are indecomposable.
          \qed
    \end{lemma}
   
  \begin{defn}
      Let $T_+$ and $T_-$ be the attracting and repelling trees of a fully irreducible free group automorphism. By Theorem \ref{thm:Q_map}, there are $F_n$-equivariant, surjective maps \[\cQ_+:\p F_n \rightarrow \hT_+ \,\, \text{     and     } \,\, \cQ_-:\p F_n \rightarrow \hT_-\] which are continuous with respect to the observers' topology. In addition, points in $\pT_+$ and $\pT_-$ have exactly one pre-image by $\cQ_+$ and $\cQ_-$, respectively. Define the dual laminations $L(T_+)$ and $L(T_-)$ as in Definition \ref{def:lam}. 
  \end{defn}

    \begin{prop} \cite[Proposition 3.22]{kapovichlustig} \label{lem:T+-int_empty}
     $L(T_+) \cap L(T_-) = \emptyset$.
      \qed
    \end{prop}

    \begin{remark}
      Let $T$ be an $\R$-tree with a very small, minimal action of $F_n$ by isometries and dense orbits, let $\cQ:\p F_n \rightarrow \hT$ be the $\cQ$-map, and let $L(T)$ be the dual lamination. The map $\cQ$ induces a map \[\cQ:L(T) \rightarrow \hT \quad \text{ by } \quad \cQ(X,Y) = \cQ(X) = \cQ(Y).\] 
      Thus, there are maps $\cQ_+\from L(T_+) \to \hT_+$ and $\cQ_-\from L(T_-) \to \hT_-$.
    \end{remark}

    Bestvina--Feighn--Handel \cite{tits0} define {\it attracting and repelling laminations} $\Lam_+ = \Lam_+^{\phi}$ and $\Lam_- = \Lam_-^{\phi}$, respectively, associated to a fully irreducible $\phi \in \out$. As we do not use the definition of these laminations explicitly, we refer the reader to \cite{tits0} for details and briefly recall the facts relevant to this paper.  Similar to the structure of the attracting and repelling trees $T_+$ and $T_-$, the attracting and repelling laminations satisfy $\Lam^{\phi}_+ = \Lam^{\phi^{-1}}_-$ and $\Lam^{\phi}_- = \Lam^{\phi^{-1}}_+$. Moreover, a strong  relationship between $\Lam_+$ and $L(T_-)$ can be seen using the following construction: 
    \[\text{diag}(\Lam_\pm) = \{ (\xi_1, \xi_m) \in \p^2F_n \mid \exists \xi_2, \dots , \xi_{m-1} \in \partial F_n \text{ with } (\xi_i, \xi_{i+1}) \in \Lam_\pm \text{ for }  1 \leq i \leq m-1 \}.\]

    \begin{lemma}\cite{chr,kapovichlustig_laminations}\label{diagLemma}
      $L(T_-) = \textup{diag}(\Lam_+)$ and $L(T_+) = \textup{diag}(\Lam_-)$. 
       \qed
    \end{lemma}

  We shall use another relationship between $\Lam_+$ and $L(T_-)$, given as follows.
    
    \begin{defn}\cite{chr} \label{defn_deriv_lam}
     Let $\cL$ be a lamination in $\doublebndry$. The {\it derived lamination of $\cL$}, denoted $\cL'$, is the set of limit points in $\cL$. That is,
    \[ \ell \in \cL' \iff \exists \{\ell_i\} \subset \cL \text{ with } \lim_{i \to \infty} \ell_i = \ell. \] 
    \end{defn}

  \begin{lemma}\cite{chr} \label{lemma_bhf_limit} 
    Let $\phi$ be a fully irreducible outer automorphism, let $\Lam_+$ denote its attracting lamination, and let $L(T_-)$ denote the dual lamination of its repelling tree. Then $\Lam_+ = L(T_-)'$. 
  \end{lemma}
  \begin{proof} 
    Let $\phi$ be a fully irreducible outer automorphism. By Lemma ~\ref{lemma_iwipindecomposable}, the action of $F_n$ on $T_-$ is indecomposable. By \cite[Theorem A]{chr}, the lamination $L(T_-)$ contains a unique minimal sublamination which equals its derived lamination $L(T_-)'$. On the other hand, the attracting lamination $\Lam_+$ is minimal and contained in $L(T_-)$. Therefore, $L(T_-)' = \Lam_+$.
  \end{proof}
  
\subsection{Realization of leaves}

  \begin{defn}
    Let $\cL$ be a lamination of $F_n$. The {\it ends of the lamination $\cL$} is the set 
    \[ \mathcal{E}\cL = \{ X \in \partial F_n \mid \exists Y \in \partial F_n \text{ such that } (X,Y) \in \cL\}.\]
  \end{defn}

  \begin{lemma}\label{endsToBndry}
    Let $\phi \in \out$ be a fully irreducible outer automorphism, let $T_{\pm}$ be the attracting and repelling trees of $\phi$, and let $\cQ_+:\p F_n \rightarrow \hT_+$ be the $\cQ$-map. For every $X \in \cE L(T_-)$, $\cQ_+(X) \in \p T_+$. 
  \end{lemma}
  \begin{proof}
    This argument is given in \cite{handelMosher_axes}. We include an outline for clarity (keeping the notation used by the reference), and omitting the definitions; consult \cite{handelMosher_axes} for more details. 
    Let $g\from \Gamma \to \Gamma$ be an affine train track representative of $\phi$, and let $\widetilde{\G}$ denote the universal cover of $\G$. Via the marking on $\G$, the boundary $\p \widetilde{\G}$ may be identified with $\p F_n$, and the attracting lamination $\Lam_+$ may be identified with a set of geodesic lines in $\widetilde{\G}$. There is an $F_n$-equivariant edge-isometry $f_g: \widetilde{\G} \rightarrow T_+$ which is an {\it $\Lam_+$-isometry}; that is, for each leaf $\ell \in \Lam_+$ viewed as an isometric embedding $\ell:\R \rightarrow \widetilde{\G}$, the map $f_g \circ \ell:\R \rightarrow T_+$ is an isometric embedding. Therefore, $f_g(\ell(\pm \infty)) := \lim_{t \to \pm \infty} f_g \circ \ell(t)$ exists and lies in $\partial T_+$. Since the limit exists, $\cQ_+(\ell(\pm \infty)) = f_g (\ell(\pm \infty))$ \cite{levittlustig}. Hence, $\cQ_+$ maps the ends of $\Lam_+$ to $\partial T_+$. By Lemma \ref{diagLemma}, $\mathcal{E}L(T_-) = \mathcal{E}\Lam_+$, so $\cQ_+(X) \in \p T_+$. 
  \end{proof}

  \begin{defn}\label{defn:ell^+}
    Let $\ell = (X,Y) \in L(T_-)$. By Lemma \ref{endsToBndry}, $\cQ_+(X), \cQ_+(Y) \in \partial T_+$, and by Proposition~\ref{lem:T+-int_empty}, $\cQ_+(X) \neq \cQ_+(Y)$. The {\it realization of $\ell$ in $T_+$} is the bi-infinite geodesic $\ell^+$ in $T_+$ connecting $\cQ_+(X)$ and $\cQ_+(Y)$. 
  \end{defn}

   Proposition~\ref{PropOfRelazationsOfLam} below compares the convergence of leaves of the lamination $L(T_-)$ in the topology on $\p^2F_n$ to the convergence of the realization of leaves in the attracting tree $T_+$ in the Hausdorff topology. 

  \begin{defn}
    If $\{\ell_i^+ \, | \, i \in \N\}$ is a sequence of bi-infinite geodesics in $T_+$ and  $\ell^+$ is a bi-infinite geodesic in $T_+$, then $\lim_{i \rightarrow \infty} \ell_i^+ = \ell^+$ in the {\it Hausdorff topology} on $T_+$ if for any subarc $I \subset \ell^+$, there exists $N \in \N$ so that $I \subset \ell_i^+$ for all $i>N$.
  \end{defn}

  \begin{defn} \label{def:unitcyl}
   Let $A$ be a basis for $F_n$. The set $\p^2 F_n$ may be identified with the space of pairs $(X,Y)$ of infinite reduced words in $A \cup A^{-1}$ with $X \neq Y$. For an infinite reduced word $X$, let $X_1$ denote the first letter of $X$. Let \[C_A :=\{(X,Y) \in \p^2F_n \, | \, X_1 \neq Y_1\}\] be the {\it unit cylinder associated to $A$}. 
  \end{defn}

  \begin{lemma}\label{lemma_ca_compact}
   Let $A$ be a basis of $F_n$. The unit cylinder associated to $A$ is open and compact in~$\p^2 F_n$. 
    \qed
  \end{lemma}
  
  \begin{prop}\label{PropOfRelazationsOfLam}
    Let $A$ a basis of $F_n$. Let $C_A  \subset \p^2F_n$ be the unit cylinder associated to $A$ and let $\{\ell_i\mid i \in \N\} \cup \{ \ell\} \subset L(T_-) \cap C_A$. Then, $\lim_{i \to \infty} \ell_i = \ell \in \p^2 F_n$ if and only if  $\lim_{i \to \infty} \ell^+_i = \ell^+ \subset T_+$ in the Hausdorff topology.
  \end{prop}
  \begin{proof}
    Let $\ell = (\xi, \eta) \in L(T_-) \cap C_A$ and $\ell_i = (\xi_i, \eta_i) \in L(T_-) \cap C_A$ so that $\lim_{i \to \infty} \ell_i = \ell$. Let $I=[a,b] \subset \ell^+$, and assume that $a$ is between $\cQ_+(\xi)$ and $b$. Let $d_1$ be the direction at $a$ containing $\cQ_+(\xi)$. Since $\cQ_+$ is continuous with respect to the observers' topology, there exists $N_1 \in \N$ so that for all $i>N_1$, $\cQ_+(\xi_i) \in d_1$. Similarly, if $d_2$ is the direction at $b$ containing $\cQ_+(\eta)$, then there exists $N_2 \in \N$ so that for all $i>N_2$, $\cQ_+(\eta_i) \in d_2$. Let $i>\max\{N_1, N_2\}$. Then $\ell_i^+$ contains $[a,b]$ as desired.

    Suppose now that $\lim_{i \to \infty} \ell^+_i = \ell^+$ in the Hausdorff topology. Since $L(T_-)\cap C_A$ is compact by Lemma~\ref{lemma_ca_compact}, the sequence $\{\ell_i\}$ has a convergent subsequence. Let $\tau$ be a partial limit; that is, $\lim_{j\to \infty}\ell_{i_j} = \tau \in \p^2F_n$. By the arguments in the previous paragraph, $\lim_{j\to \infty}\ell_{i_j}^+ = \tau^+$ in the Hausdorff topology on $T_+$. Thus, $\ell^+ = \tau^+$. By Theorem~\ref{thm:Q_map}, the map $\cQ_+$ is injective on $\cE L(T_-)$, so $\tau = \ell$. Therefore, the sequence $\{\ell_i\}$ has a unique partial limit; thus, $\{\ell_i\}$ converges to~$\ell$. 
  \end{proof}

  We conclude this section by describing a property of the realization of leaves of $L(T_-)$ in $T_+$ that we will use in Section \ref{sec_k33s}.

  \begin{defn}
    A {\it star} is a wedge of intervals or a wedge of  rays. The wedge point is called the \emph{middle} of the star.
  \end{defn}

  \begin{prop}\cite{handelMosher_axes}\label{fiberIsStar}
    Let $\ell_1, \ell_2, \dots, \ell_k$ be leaves of the lamination $L(T_-)$ such that $\ell_i$ is asymptotic to $\ell_{i+1}$ for each $i=1, \dots, k-1$. Then $\bigcup_{i=1}^k \ell_i^+$ is a star in $T_+$.
  \end{prop}
  
  To prove this property we will need some facts from the next subsection.

\subsection{Automorphisms, branch points, and homotheties}

  We will use the following facts about the correspondence between a fully irreducible outer automorphism $\phi$ and an automorphism in $\Aut(F_n)$ representing $\phi$. 
  
  \begin{defn}
    Let $R_n$ denote the graph with one vertex $v$ and $n$ edges. Choosing a basis $A$ of $F_n$ and identifying each oriented edge with a distinct element of $A$ identifies $\pi_1(R_n, v)$ with $F_n$. Moreover, each automorphism $\Phi \in \aut$ is represented by a map $f \from R_n \to R_n$ sending the vertex $v$ to itself and an edge of $R_n$ to an immersed edge path so that $f$ is a homotopy equivalence. The correspondence between such self-maps of $R_n$ and elements of $\aut$ is a bijection. Let $G$ be a graph and $\mu \from R_n \to G$ a homotopy equivalence. A homotopy equivalence $f \from G \to G$ gives rise to an outer class $[f_*] \in \out$. If $f(e)$ is an immersed edge path for every edge $e \in E(G)$, then we say that $f$ is a \emph{topological representative} of $[f_*]$. Let $\phi \in \out$, let $f \from G \to G$ be a topological representative of $\phi$, and let $\wt G$ be the universal cover of $G$. The identification, via the homotopy equivalence $\mu$, of $\pi_1(G,*)$ with $F_n$ (up to conjugation) gives rise to an action of $F_n$ on $\wt G$ by deck transformations $\rho \from F_n \to \text{Aut}(\wt G)$. The map $\rho$ is well defined up to precomposing it with $i_g$, where $i_g \in \aut$ denotes conjugation by $g$. 
  \end{defn}  

  \begin{lemma} \label{lemma_correspondence}
    Let $\phi \in \out$ and $f \from G \to G$ be a topological representative of $\phi$. Let $\wt G$ denote the universal cover of $G$ and $\rho \from F_n \to \text{Aut}(\wt G)$ the action of $F_n$ by deck transformations. There is a bijection  between automorphisms $\Phi \in \phi$ and lifts $\wt g \from \wt G \to \wt G$ of $g$ given by the equation:
    \begin{equation}
      \wt g \circ \rho(\gamma) = \rho( \Phi(\gamma) ) \circ \wt g.
    \end{equation}
	  \qed
  \end{lemma}

  \begin{defn}\label{def_principal}
    Let $\phi \in \Out(F_n)$ be fully irreducible, and let $\Phi \in \Aut(F_n)$ so that $\Phi \in \phi$. We say that $\Phi$ is a {\it principal automorphism} if the extension of $\Phi$ to $\p F_n$, which we denote by $\partial \Phi$, fixes at least three nonrepelling points. 
  \end{defn}

  \begin{lemma} \label{lem_prin_aut} \cite{GJLL,handelMosher_axes}
    Let $\phi \in \out$ be fully irreducible, and let $T_+$ be its attracting tree. 
    \begin{enumerate}
      \item If $\Phi$ is a principal automorphism representing $\phi$, then there exists a unique homothety $h_+ \from T_+ \to T_+$ so that $h_+( \gamma x) = \Phi(\gamma) h_+(x)$ for all $\gamma \in F_n$ and $x \in T_+$. This homothety stretches the distances in $T_+$ by the dilatation of $\phi$. 
      \item The homothety $h_+$ from the previous item fixes a branch point $b \in T_+$.
      \item The correspondence $\Phi \mapsto b$ defines a bijection between the set of principal automorphisms representing $\phi$ and the set of branch points of $T_+$. 
      \item For each automorphism $\Phi$ representing $\phi$ there exists a homothety $h_- \from T_- \to T_-$ so that $h_-( \gamma x) = \Phi(\gamma) h_-(x)$ for all $\gamma \in F_n$  and for all $x\in T_-$. 
     \end{enumerate}
  \end{lemma}

  \begin{proof}
    Item (4) is the only one that does not appear in \cite{handelMosher_axes}. Since $\phi$ acts on $\overline{CV}_n$ by North-South dynamics, $T_-$ is equivalent to $T_- \cdot\phi$. This means that there exists a homothety $h_- \from T_- \to T_-$ which is $F_n$-equivariant. That is, if $\rho \from F_n \to \text{Aut}(T_-)$ is the $F_n$-action on $T_-$ then,
    \[ h_-( \rho(\gamma) x) = \rho(\Phi(\gamma)) h_-(x).  \]
    Suppressing $\rho$ yields the statement in the lemma.
  \end{proof}

  It will be easier to work with rotationless automorphisms and topological representatives. One can find the definitions of rotationless automorphisms and rotationless train track maps in \cite{fh09} or \cite[pages 24,27]{handelMosher_axes}. Here we will only need the following facts due to Feighn-Handel \cite{fh09}.
  \begin{enumerate}
    \item Each fully irreducible $\phi \in \out$ has a rotationless power $\phi^p$, and $p$ is bounded by a constant depending only on $n$. Moreover, $T_\pm ^\phi = T_\pm ^{\phi^p}$ and $\Lambda_\pm ^\phi = \Lambda_\pm ^{\phi^p}$. 
    \item Let $\phi$ be a fully irreducible automorphism and $f \from G \to G$ a train track representative of $\phi$. Then, $\phi$ is rotationless if and only if $f$ is rotationless.
    \item If $f \from G \to G$ is rotationless, then for each vertex $v$, $f^2(v) = f(v)$. For each directed edge $e$, the germ of $f^2(e)$ equals that of $f(e)$.  
  \end{enumerate} 

  \begin{lemma}\label{asymptoticImpliesFixed}
    Let $\phi \in \out$ be fully irreducible. There exists $\phi^p$, a rotationless power of $\phi$, so that for any asymptotic leaves $\ell_1, \ell_2 \in \Lam_+^\phi$, there exists an automorphism $\Phi \in \Aut(F_n)$ representing $\phi^p$ so that the endpoints of $\ell_1$ and $\ell_2$ are non-repelling fixed points of $\partial\Phi$. 
  \end{lemma}

  \begin{proof}
    We call a triplet $(X,Y,Z)$ of distinct points in $\partial F_n$ \emph{special} if $(X,Y), (Y,Z) \in \Lam_+^\phi$. The set of special triplets is $F_n$-invariant. Moreover, if $(X,Y,Z)$ is special, then by Lemma~\ref{diagLemma}, $\cQ_-(X) = \cQ_-(Y) = \cQ_-(Z)$. Since the $\cQ_-$-index is bounded (by $2n-2$) \cite{coulboishilion} (see \cite{coulboishilion} for the definition), there are finitely many $F_n$-orbits of special triplets. Therefore, there exist  $q,m \in \N$ so that for every automorphism $\Psi$ representing $\phi$, and for each special triplet $W = (X,Y,Z)$, the special triplets $\Psi^{m+q}(W)$ and $\Psi^q(W)$ are in the same $F_n$-orbit. 

    Let $\phi^p$ be a rotationless power of $\phi$. Let $f \from G \to G$ be a train-track representative of $\phi^p$.  Identify $\partial \wt G$ with $\partial F_n$, let $f_0 \from \wt G \to \wt G$ be any lift of $f$, and let $\Phi_0$ be the corresponding automorphism given by Lemma~\ref{lemma_correspondence} representing $\phi^p$. By the previous paragraph, there exists $g \in F_n$ so that each special triplet $W=(X,Y,Z)$ satisfies $\partial f_0^{m+q}(W) = g \partial f_0^q(W)$. The map $f$ is rotationless; hence, the image of each germ is fixed by $f$. So, $m=1$.  Moreover, we may assume $q=0$ or $q=1$. If $q=0$, then $\pf (W) = g W$. Let $\Phi = i_{g^{-1}} \circ \Phi_0$, where $i_g$ denotes conjugation by $g$. Then, $\partial \Phi (W) = W = (X,Y,Z)$. The points $X,Y,Z$ are attracting fixed points in this case, since they belong to $\mathcal{E}\Lam_+^\phi$. If $q=1$, then $\pf^2(W) = g \pf(W)$. Thus, $W = \Phi_0^{-2}(g) \partial\Phi_0^{-1}(W)$. Let $\Phi^{-1} = i_w \circ \Phi_0^{-1}$ for $w =  \Phi_0^{-2}(g)$. Then, $W$ is fixed by $\Phi^{-1}$, which represents $\phi^{-p}$. Therefore, $\Phi$ also fixes $W = (X,Y,Z)$, and the points $X,Y,Z$ must be attracting fixed points, since they belong to $\mathcal{E}\Lam_+^\phi$.
  \end{proof} 

  \begin{lemma} \cite[Equation 4.2]{coulboishilion_bot}
    \label{Qandh}
    Let $\Phi$ be an automorphism and $h_\pm \from T_\pm \to T_\pm$ the corresponding homotheties guaranteed by Lemma \ref{lem_prin_aut}. Then, for each $\xi \in \partial F_n$,
    \[ \cQ_\pm( \partial \Phi(\xi)) = h_\pm(\cQ_\pm(\xi)). \]
    \qed
  \end{lemma}

  \begin{proof}[Proof of Proposition \ref{fiberIsStar}]
     By Lemma \ref{asymptoticImpliesFixed} there exists a principal automorphism $\Phi$ representing $\phi^p$ for some $p \in \mathbb{N}$ so that the endpoints of $\ell_1$ and $\ell_2$ are attracting fixed points of $\partial\Phi$. Similarly, since $\ell_2$ and $\ell_3$ are asymptotic, there exists an automorphism $\Phi'$ representing $\phi^p$ so that the endpoints of $\ell_2$ and $\ell_3$ are attracting fixed points of $\partial\Phi'$.
     Attracting fixed points of two automorphisms in the same outer class are disjoint \cite [Theorem 1.1]{hilion_stab} (see alternatively \cite[Corollary 2.9]{handelMosher_axes} for two principal automorphisms), so $\Phi = \Phi'$. Continuing in this fashion proves that the end-points of $\ell_1, \dots , \ell_k$ are attracting fixed points of $\partial \Phi$. By Lemma \ref{lem_prin_aut}(1), there exists a homothety $h \from T_+ \to T_+$ representing $\Phi$. Let $S \subset T_+$ be the union of the realizations of $\{\ell_i\}_{i=1}^k$. Then, $S$ is invariant by $h$ by Lemma~\ref{Qandh}. Moreover, $S$ has finitely many vertices, and, $h$ permutes the finitely many vertices of $S$. Since $h$ is a homothety, $S$ contains only one vertex.  
      \end{proof}

\subsection{Proof of Theorem~\ref{thm:Whitehead graph}} \label{sec_wh_graphs_iso}

     Let $\Wh_{\phi}$ be the $F_n$-quotient of the graph whose vertex set is the union of nonrepelling fixed points in $\p F_n$ of principal automorphisms representing $\phi$; see Definition~\ref{def_principal}. An edge of $\Wh_{\phi}$ corresponds to a leaf of the attracting lamination $\Lam_+^\phi$ of $\phi$. Let $\Wh_{\Lam_+^\phi}$ be the $F_n$-quotient of the graph whose vertex set is the union of endpoints of singular leaves of $\Lam_+^\phi$, where a singular leaf has an asymptotic class containing more than one element. An edge of $\Wh_{\Lam_+^\phi}$ corresponds to a singular leaf of $\Lam_+^\phi$. By Lemma \ref{asymptoticImpliesFixed}, $Wh_{\phi} \cong \Wh_{\Lam_+^\phi}$, proving Theorem~\ref{thm:Whitehead graph}.

\subsection{System of partial isometries associated to a free group automorphism} \label{subsec:trees_Q}
   
    Coulbois--Hilion--Lustig in \cite{coulboishilionlustig} use the $\cQ$-map to construct a system of partial isometries for an $\R$-tree $T$ with a very small, minimal action of $F_n$ by isometries and dense orbits. 
    
    \begin{construction} \label{const:system}  \cite{coulboishilionlustig}.
     Let $T$ be a very small, minimal $F_n$-tree with dense orbits, let $A$ be a basis of $F_n$, and let $C_A \subset \p^2 F_n$ be its unit cylinder. Let    
     \[\Omega_A := \cQ(L(T) \cap C_A) \subset \oT.\]
     The {\it compact heart $K_A$ of $T$ relative to $A$} is the convex hull of $\Omega_A$ in $\oT$. By \cite{coulboishilionlustig}, $K_A$ is indeed a compact subtree of $\oT$. Let $S = (K_A, A)$   be the system of partial isometries so that for each $a \in A$, the partial isometry associated to $a$ is the maximal restriction of $a^{-1}$ to $K_A$. As defined in Section \ref{subsec:systems}, let $\cB$ be the associated band complex of $S$, and let $T_{\cB}$ denote the $F_n$-tree dual to $\cB$. 
    \end{construction}
    
    \begin{remark}\label{translatingSysPartialIsom}
    If $K_A$ is the compact heart of $T$ relative to the basis $A$ of $F_n$, then $gK_A$ is the compact heart of $T$ relative to the basis $gAg^{-1}$ of $F_n$. Moreover, the unit cylinder satisfies $C_{gAg^{-1}} = gC_A$ and $\Omega_{gAg^{-1}} = g \Omega_A$. We leave the verification of these facts to the reader. 
        \end{remark}

    \begin{thm} \cite{coulboishilionlustig}
    \label{compactHeart}
    Let $T$ be a minimal, very small, $F_n$-tree with dense orbits, and let $A, C_A, \cB,$ and $T_\cB$ be defined as in Construction \ref{const:system}. 
     \begin{enumerate}
      \item The tree $T$ is equal to the minimal subtree of the tree $T_\cB$.      
      \item $L(T) \cap C_A = L(\cB)$, where $L(T) \cap C_A$ and $L(\cB)$ are identified with bi-infinite reduced words in $A\cup A^{-1}$ as in the above construction. 
		 \qed
	     \end{enumerate}
    \end{thm}
    
  \begin{remark}
    When we apply Construction \ref{const:system} and Theorem \ref{compactHeart} to the attracting or repelling trees $T_+$ or $T_-$ of a fully irreducible outer automorphism, we denote each object with the appropriate subscript; we write $S_+, \Omega_+, \ldots$, and so on. 
  \end{remark}   

\section{Rips Induction and overlapping bands} \label{subsec:rips_class}

\subsection{Rips Induction}

  \begin{defn}
    Let $S = (K,A)$ be a system of partial isometries. The {\it output of the Rips Machine applied to $S$} is a new system of partial isometries $S' = (K', A')$ defined as follows. 
    \[ K':= \{ x \in K \, | \, x \in \dom(a) \cap \dom(a') \text{ for some } a \neq a' \in A \cup A^{-1}  \}.  \]
    Since $A$ is finite and the intersection of two domains is a compact $\R$-tree, $F'$ is a compact forest. Let $A'$ be the set of all maximal restrictions of the elements of $A$ to pairs of components of $F'$. Then $S' = (F', A')$ is a system of partial isometries. 
  \end{defn}

  \begin{defn} \cite[Definition 3.11]{coulboishilion}
    Let $S_0 = (K_0, A_0)$ be a system of partial isometries, and let $S_1 = (K_1, A_1)$ denote the output of the Rips Machine. The system of partial isometries $S_0$ is {\it reduced} if for any partial isometry $a \in A_0^{\pm}$, the set of extremal points of the domain of $a$ is contained in $K_1$. 
   \end{defn}

  \begin{prop} \label{prop_reduced} \cite[Propositions 5.6, 3.14]{coulboishilion}
    Let $\phi \in \Out(F_n)$ be a fully irreducible outer automorphism,  and let $S_+ = (K_A^+, A)$ and $S_- = (K_A^-,A)$ be the systems of partial isometries defined in Construction \ref{const:system}. Then $S_+$ and $S_-$ are reduced, as is any output of $S_+$ or $S_-$ under the Rips Machine.    \qed
  \end{prop}

  \begin{defn}
    Let $S_0 = (K_0, A_0)$ be a system of partial isometries. Let $S_i= (K_i, A_i)$ denote the output of the $i^{th}$ iteration of the Rips Machine. If for some $i$, $K_i = K_{i+1}$, then the Rips Machine {\it halts} on $S_i$, and the Rips Machine {\it eventually halts} on $S_0$. 
  \end{defn}

  \begin{defn} \label{def_levitt_type}
    Let $S_0$ be a system of partial isometries. If the Rips Machine eventually halts on $S_0$, then $S_0$ is called {\it surface type}. If the Rips Machine does not eventually halt on $S_0$ and $$\lim_{i\rightarrow \infty}\max_{a \in A_i} \diam a = 0,$$ then $S_0$ is called {\it of Levitt type.} 
  \end{defn}
       
    The definition of `Levitt type' given here is equivalent to the original one in \cite[Section 5]{coulboishilion_bot}, which states that the limit set $\Omega$, see Definition~\ref{def_systems}, is totally disconnected.
       
  \begin{notation}
    Let $\phi \in \Out(F_n)$ be a fully irreducible outer automorphism, and let $S_+$ and $S_-$ be the systems of partial isometries defined in Construction \ref{const:system}. If $S_+$, respectively $S_-$, is of Levitt type, we say $T_+$, respectively $T_-$, is {\it of Levitt type.} 
  \end{notation}       

  The following proposition is a direct consequence of a result of Coulbois--Hilion \cite{coulboishilion_bot}.

  \begin{prop} \cite[Theorem 5.2]{coulboishilion_bot} \label{prop:Levitt}
    Let $\phi \in \Out(F_n)$ be a fully irreducible atoroidal outer automorphism.  Then either both $T_+$ and $T_-$ are of Levitt type, or one of the trees $T_+$ and $T_-$ is of Levitt type and the other one is of surface type.
        \qed
  \end{prop}
      
  The proof of our Main Theorem uses the fact that at least one of $T_+$ and $T_-$ is of Levitt type.

 \subsection{Volume of an $\R$-tree and overlapping domains}
 
  The main aim of this section is to prove Proposition \ref{prop:levitt implies overlap}, which is a converse to \cite[Proposition 4.3]{coulboishilion}, and states that if $T$ is an $F_n$-tree with dense orbits and of Levitt type, then the associated Rips machine (beginning from the compact heart) has the property that at each step there are three overlapping bands.
      
  \begin{defn}\label{def:volume}
     A compact $\R$-tree $K$ is {\em finite} if $K$ has a finite number of extremal points. In this case, $K$ has also a finite number of branch points. Removing these branch points from $K$ yields to a finite set of arcs; the {\em volume} $\vol(K)$ of $K$ is the sum of the lengths of these arcs. The volume $\vol(K)$ of a compact $\R$-tree $K$ is the supremum of the volume of the finite subtrees contained in $K$. A {\em compact forest} $K$ is a finite disjoint union of compact trees $K_1,\dots,K_p$; its {\em volume} is $\vol(K)=\sum_{1\leq i \leq p}\vol(K_i)$.
 \end{defn}
 
    This volume of an $\R$-tree may be finite or infinite. In either case, the following proposition applies. 

  \begin{prop}\label{prop:compact tree}
   Let $T$ be a compact $\R$-tree.
   Let $\epsilon>0$ and 
   let $\mathcal E$ be a set of disjoint arcs in $T$ of length $\epsilon$.
   Then $\mathcal E$ is finite.
  \end{prop}
  
  For the proof of Proposition~\ref{prop:compact tree}, we will use the following combinatorial lemma (that we state in the context of $\R$-trees even if it is valid for more general notions of trees).
  
  \begin{lemma}\label{lemma:pwdisjoint}
   Let $T$ be an $\R$-tree.
   Let $\mathcal E$ be an infinite set of pairwise disjoint subtrees of~$T$.
   There exists a sequence $(T_n)_{n \in \N}$ of pairwise distinct elements of $\cE$ with the property that, for all $n$, $T_{n+1}$ does not intersect the convex hull $K_n$ of $\bigcup_{i=1}^n T_i$.
  \end{lemma}  
  \begin{proof}
  First, suppose there exists an arc $\gamma$ in $T$ intersecting an infinite number of elements of $\mathcal E$.
  Choose an orientation on $\gamma$, which induces a linear ordering $\prec$ on the set of elements of $\mathcal E$ that intersect $\gamma$. Any countable subset of $\cE$ gives rise to a sequence $(T_n)_{n\in\N}$ of pairwise disjoint subtrees of $T$ such that for all $n\in\N$, $T_n\prec T_{n+1}$. In particular $T_{n+1}$ does not intersect the convex hull of $\bigcup_{i=1}^n T_i$.
  
  In the remaining case, every arc of $T$ intersects a finite number of elements of $\mathcal E$. We build the sequence $(T_n)_{n\in\N}$ by induction. First, pick an element $T_1$ of $\mathcal E$. Suppose $T_1,\dots,T_n$ have been chosen so that no element of $\cE - \bigcup_{i=1}^n T_i$ intersects the convex hull $K_n$ of $\bigcup_{i=1}^n T_i$. Let $S$ be an element of $\cE - \bigcup_{i=1}^n T_i$. By assumption, $S$ does not intersect $K_n$. Let $\gamma$ be the arc in $T$ joining $S$ and $K_n$. Set $T_{n+1}$ to be the element of $\mathcal E$ closest to $K_n$ that intersects~$\gamma$.
  \end{proof}
  
  \begin{proof}[Proof of Proposition~\ref{prop:compact tree}]
  Assume towards a contradiction there is a sequence $(\gamma_n)_{n\in\N}$ of disjoint arcs in $T$ of length $\epsilon>0$. By Lemma~\ref{lemma:pwdisjoint}, we can assume that up to reordering indices, $\gamma_{n+1}$ does not lie in the convex hull $K_n$ of $\bigcup_{i=1}^n \gamma_i$.
  Since $\gamma_{n+1}$ has length $\epsilon$, at least one of its extremal points, denoted by $x_{n+1}$, is at distance more than $\epsilon/2$ from $K_n$.
  Thus, for all $i\neq j\in\N$, $d(x_i,x_j)>\epsilon/2$.
  Hence, no subsequence of $(x_k)_{k\in\N}$ is convergent (since such a subsequence is not a Cauchy sequence), which contradicts the compactness of~$T$.
  \end{proof}

  \begin{remark}\label{remark:vol=0}
  Let $K$ be a (non-empty) compact tree. Then $\vol(K)=0$ if and only if $K$ is point. Indeed, if $K$ contains two distinct points, then $K$ contains the arc joining these two points, and hence the volume of $K$ is as least as big as the length of this interval.
  \end{remark}
  
  \begin{notation}
   Let $S = (K,A)$ be a system of partial isometries. Let $A^{-1} = \{a^{-1} \, | \, a \in A\}$, and let $A^\pm = A \cup A^{-1}$. We suppose from now on that $A \cap A^{-1} = \emptyset$.  
   Let $v(x)$ be the band valence of $x$; that is, $v(x) = \#\{a \in A^\pm \mid x \in \dom(a) \}$. 
   Let $K^{=i} = \{ x \in K \mid v(x) = i \}$. Define $K^{\geq i}$ and $K^{\leq i}$ similarly.
  \end{notation}

  \begin{lemma}\label{lemma:K>3 finite} 
  Let $S=(K,A)$ be a system of partial isometries, such that $\vol\left(K^{\geq 3}\right)=0$. Then:
  \begin{enumerate}[(i)]
  \item $K^{\geq 3}$ is a finite set of points in $K$.
  \item If, moreover, the valence of every point in $K$ is finite, then the set 
   $\{\delta \text{ direction at } x \;|\;  x\in K^{\geq 3} \}$ of directions in $K$ at a point in $K^{\geq 3}$ is finite.
  \end{enumerate}
  \end{lemma}
  \begin{proof}
  Property (ii) follows immediately from Property (i).  
  According to Remark~\ref{remark:vol=0}, since $\vol\left(K^{\geq 3}\right)=0$, the components of $K^{\geq 3}$ are points. It remains to show there are only finitely many components.
  Since the domains of the partial isometries in $A^{\pm1}$ are subtrees of the tree $K$, an intersection of domains is also a tree (and possibly the empty set). Thus, the set of components of $K^{\geq 3}$ injects in the set of subsets of $A^{\pm1}$ of cardinality at least 3, which is a finite set since $A^{\pm1}$ is itself a finite set.
  \end{proof}
  
  Let $T$ be a free $\R$-tree with dense orbits and of Levitt type, and let $A$ be a basis of $F_n$.
  Let $K_0$ be the compact heart of $T$ relative to $A$, and let $S_0 = (K_0, A_0)$ be the associated system of partial isometries as in Construction~\ref{const:system}.
  Since $T$ is of Levitt type, the Rips machine does not halt. Let $S_i= (K_i, A_i)$ denote the output after the $i^{th}$ iteration of the Rips Machine. Then, $\vol(K_i^{=1})>0$ for all $i\in\N$.

  \begin{remark}\label{remark:no finite orbit}
  \begin{enumerate}[(i)]
  \item All orbits of $S_0 = (K_0, A_0)$ are infinite; that is, every leaf in the band complex $\cB$ is infinite. Indeed, \cite[Proposition~5.6]{coulboishilion} ensures the system $S_0$ is reduced, which, by definition, implies every orbit is infinite.
  \item The valence of the points in $K_0$ is bounded. Indeed, $K_0$ is a subtree of $\overline{T}$, and the valence of the points of $\overline T$ is bounded (by $2n$ \cite{GJLL}).
  \end{enumerate}
  \end{remark}
  
  \begin{prop}\label{prop:levitt implies overlap}
  Let $T$ be an $F_n$-tree with dense orbits and of Levitt type, and let $A$ be a basis of $F_n$.
  Let $S_0 = (K_0, A_0)$ be the associated system of partial isometries, and let $S_i= (K_i, A_i)$ denote the output after the $i^{th}$ iteration of the Rips Machine.
  Then $\vol\left(K_i^{\geq 3}\right)>0$ for all $i\in\N$.
  \end{prop}
  \begin{proof}
  By contradiction, suppose $\vol\left(K_{i_0}^{\geq3}\right)=0$ for some $i_0\in\N$.
  By definition of the Rips machine, $K_{i+1}^{\geq 3}\subseteq K_{i}^{\geq 3}$.
  In particular, the sequence $\left(\vol\left(K_i^{\geq 3}\right)\right)_i$ is decreasing, and thus $\vol\left(K_{i}^{\geq3}\right)=0$ for all $i\geq i_0$. We can suppose that $i_0=0$.

  We define by induction a sequence $(\gamma_n)$ of subarcs of $K_0$.
  Let $\gamma_0$ be an arc contained in $K_0^{=1}$.
  There is only one partial isometry $a$ defined on $\gamma_0$; let $\gamma_1=\gamma_0\cdot a$, the image of $\gamma_0$ by~$a$. Since $S_0$ has no finite orbit, $\gamma_1 \subseteq K_0^{\geq 2}$. In particular, $\gamma_0$ and $\gamma_1$ are disjoint.

  {\em \underline{Case 1}:} $\gamma_1 \subseteq K_0^{=2}$. There are exactly two partial isometries defined on $\gamma_1$: one is $a^{-1}$; let $b$ denote the other. Set $\gamma_2=\gamma_1\cdot b$. Again, $\gamma_2 \subseteq K_0^{\geq 2}$ since $S_0$ has no finite orbits.

  {\em \underline{Case 2}:} $\gamma_1\cap K_0^{\geq 3}\neq\emptyset$. Replace $\gamma_0$ by a subarc so $\gamma_1\subseteq K_0^{=2}$; this procedure is possible since $K_0^{\geq 3}$ is a finite subset of $K_0$ by Lemma~\ref{lemma:K>3 finite}. Define $\gamma_2$ as in Case 1.

  Iterate this process. Since the valence of the points of $K$ is finite (see Remark~\ref{remark:no finite orbit}),  statement (ii) of Lemma~\ref{lemma:K>3 finite} ensures that after a finite number of iterations, Case 2 stops to occur.

  Finally, the arcs $\gamma_n$ are pairwise disjoint. Indeed, $\gamma_0 \subseteq K_0^{=1}$ and $\gamma_i \subseteq K_0^{=2}$ for all $i\in\N$. (In other words, one can perform a sequence of ``Rips moves'' successively based on $\gamma_0$, $\gamma_1$, $\gamma_2$, and so on.) This collection contradicts Proposition~\ref{prop:compact tree}.
\end{proof}

    By \cite[Proposition~4.3]{coulboishilion}, if $S_0 = (K_0, A_0)$ is a pseudo-surface system of partial isometries, then  $\vol\left(K_0^{\geq 3}\right)=0$. As shown in \cite[Proposition 5.14]{coulboishilion}, a free mixing $F_n$-tree $T$ is either pseudo-surface -- in which case $\vol\left(K_0^{\geq 3}\right)=0$ by \cite[Proposition~4.3]{coulboishilion} -- or Levitt type -- in which case $\vol\left(K_i^{\geq 3}\right)>0$ for all $i\in\N$ by Proposition~\ref{prop:levitt implies overlap}. Combining these two propositions yields the following corollary.
    
  \begin{cor}\label{cor:type}
  Let $T$ be a free mixing $F_n$-tree and let $A$ be a basis of $F_n$.
  Let $S_0 = (K_0, A_0)$ be the associated system of partial isometries, and let $S_i= (K_i, A_i)$ denote the output after the $i^{th}$ iteration of the Rips Machine. Then
  \begin{itemize}
  \item $T$ is pseudo-surface if and only if $\vol\left(K_i^{\geq 3}\right)=0$ for some $i\in\N$,
  \item $T$ is Levitt type if and only if $\vol\left(K_i^{\geq 3}\right)>0$ for all $i\in\N$.
  \end{itemize}
   \qed
  \end{cor}

\section{Cannon--Thurston maps} \label{sec:CTmaps}

  The existence of a Cannon--Thurston map in the setting of hyperbolic free-by-cyclic groups was shown by Mitra (Mj) \cite{mitra}. The structure of Cannon--Thurston maps for hyperbolic free-by-cyclic groups which do not virtually split over $\Z$ was investigated by Kapovich--Lustig \cite{kapovichlustig}. We will use the following results. 
  
  \begin{thm} \cite{mitra}\label{thm:CT}
   Let $G_{\phi} = F_n \rtimes_{\phi} \Z$ be a hyperbolic group. There exists a continuous surjection $\hi: \p F_n \rightarrow \p G_{\phi}$. 
   The map $\hi$ is the {\it Cannon--Thurston map}.
    \qed
  \end{thm}
   
    \begin{thm} \label{thm:CT_not_inj} \cite{mitralaminations, kapovichlustig}
   Suppose $(X,Y) \in \p^2F_n$, and let $\hi:\p F_n \rightarrow \p G_{\phi}$ be the Cannon--Thurston map. Then, $\hi(X) = \hi(Y)$ if and only if $(X,Y) \in L(T_+) \cup L(T_-)$. 
    \qed
  \end{thm}
  
  \begin{prop} \label{prop:CT_splits} \cite[Proposition 4.8, Lemma 4.9]{kapovichlustig}
   The Cannon--Thurston map \[\hi: \p F_n \rightarrow \p G_{\phi}\] factors through the maps 
    \[\cQ_+:\p F_n \rightarrow \hT_+ \,\, \text{ and } \,\,\cQ_-:\p F_n \rightarrow \hT_-,\] and thus induces well-defined maps 
    \[\cR_+:\hT_+ \rightarrow \p G_{\phi} \,\, \text{ and } \,\, \cR_-:\hT_- \rightarrow \p G_{\phi}\] which are surjective and $F_n $-equivariant. Furthermore,
   \begin{enumerate}
    \item[(a)] $\cR_+(\oT_+) \cap \cR_+(\hT_+ \backslash \oT_+) = \emptyset$. \\
      Likewise, $\cR_-(\oT_-) \cap \cR_-(\hT_- \backslash \oT_-) = \emptyset$. 
    \item[(b)]  $\cR_+(\oT_+) \cap \cR_-(\oT_-) = \emptyset$. 
    \item[(c)] The restriction $\cR_+|_{\oT_+}$ of $\cR_+$ to the metric completion of $T_+$ is injective. \\
      Likewise, the restriction $\cR_-|_{\oT_-}$ of $\cR_-$ to the metric completion of $T_-$ is injective.   \qed
   \end{enumerate}
  \end{prop}
  
\section{The directional Whitehead graph} \label{sec_dir_pat}
    
    In this section, we introduce a tool called the {\it directional Whitehead graph} to study certain finer asymptotic relations between singular leaves. We define the graph in two ways: first, using the $\cQ$-map and its dual lamination, and second, using the band complex associated to a system of partial isometries and its dual lamination.   
    For fully irreducible outer automorphisms, there is a well-studied correspondence between these objects, and we prove in Lemma \ref{lemma_WH_same} that these two definitions agree in this setting. 
    Our main theorem hinges on a certain property of the directional Whitehead graphs of $T_-^\phi$. 

\subsection{Directional Whitehead graph} \label{subsec_dir_wh}

    \begin{defn} \label{def_dirWH}
      Let $T$ be an $\R$-tree with a very small, minimal action of $F_n$ by isometries and dense orbits, and let $\cQ:\p F_n \rightarrow T$ be the $\cQ$-map given in Theorem \ref{thm:Q_map}. Let $L(T)$ be the lamination dual to $T$. Let $x \in T$ be a branch point, and let $d$ be a component of $T \setminus \{x\}$. The \emph{directional Whitehead graph of $T$ at $d$}, denoted $\Wh_T(x,d)$, is defined as follows. There is an edge $(Y,Y')$ in the graph if there exists a  leaf $\ell = (Y,Y') \in L(T)$ such that $\cQ(\ell) = x$ and for which there exists a sequence $\{\ell_i\}_{i=1}^\infty \subset L(T)$ limiting to $\ell$ and $\cQ( \ell_i) \in d$ for all $i$. Edges share a vertex if the corresponding leaves are asymptotic.
    \end{defn}

   \begin{remark}
     In this paper, the only property of the directional Whitehead graph that we are interested in is whether the graph contains more than one edge. An interesting problem is to investigate other properties of the graph such as connectivity.
  \end{remark}

    The edges in the directional Whitehead graph, by definition, correspond to leaves which are limits of other leaves; that is, every edge of a directional Whitehead graph corresponds to a leaf in the {\it derived lamination}; see Definition~\ref{defn_deriv_lam}. Recall Lemma \ref{lemma_bhf_limit}: if $\phi \in \out$ is fully irreducible, then $L(T_-)' = \Lam_+^\phi$. 

    The second definition of the directional Whitehead graph is given in terms of systems of partial isometries. The second definition is used in Lemma \ref{DWvsTLevitt}.

  \begin{defn} \label{def_dirWH_band}
    Let $(K,A)$ be a system of partial isometries, and let $\cB$ be the corresponding band complex. Let $x \in K$, and let $d$ be a direction of $K$ at $x$. The {\it directional Whitehead graph of $d$ with respect to $(K,A)$}, denoted $\Wh_{(K,A)}(x,d)$, is defined as follows. There is an edge $(\xi,\eta)$ in the graph $\Wh_{(K,A)}(x,d)$ if there exists a leaf $(\xi,\eta) \in L(\cB)$   based at $x$ and sequences of leaves $\{(\xi_i,\eta_i)\}_{i \in \N}$ based at $x_i \in d$ for all $i$ and so that $\xi_i \to \xi$, $\eta_i \to \eta$. See Figure \ref{fig_bands_WH} for an illustration.  \end{defn}

  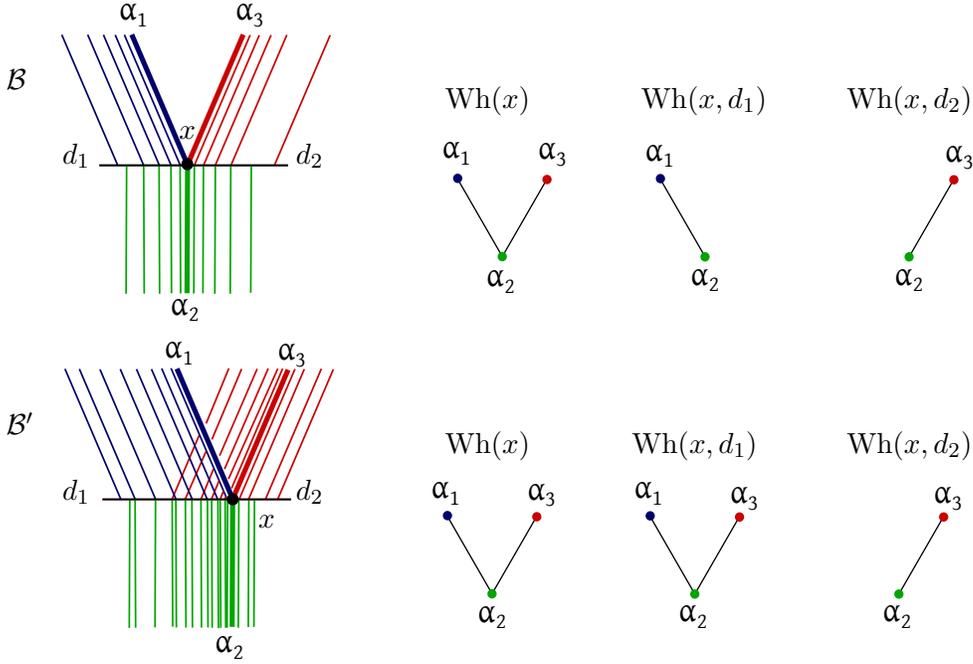
\begin{figure}
     \begin{overpic}[scale=.8, tics=5]{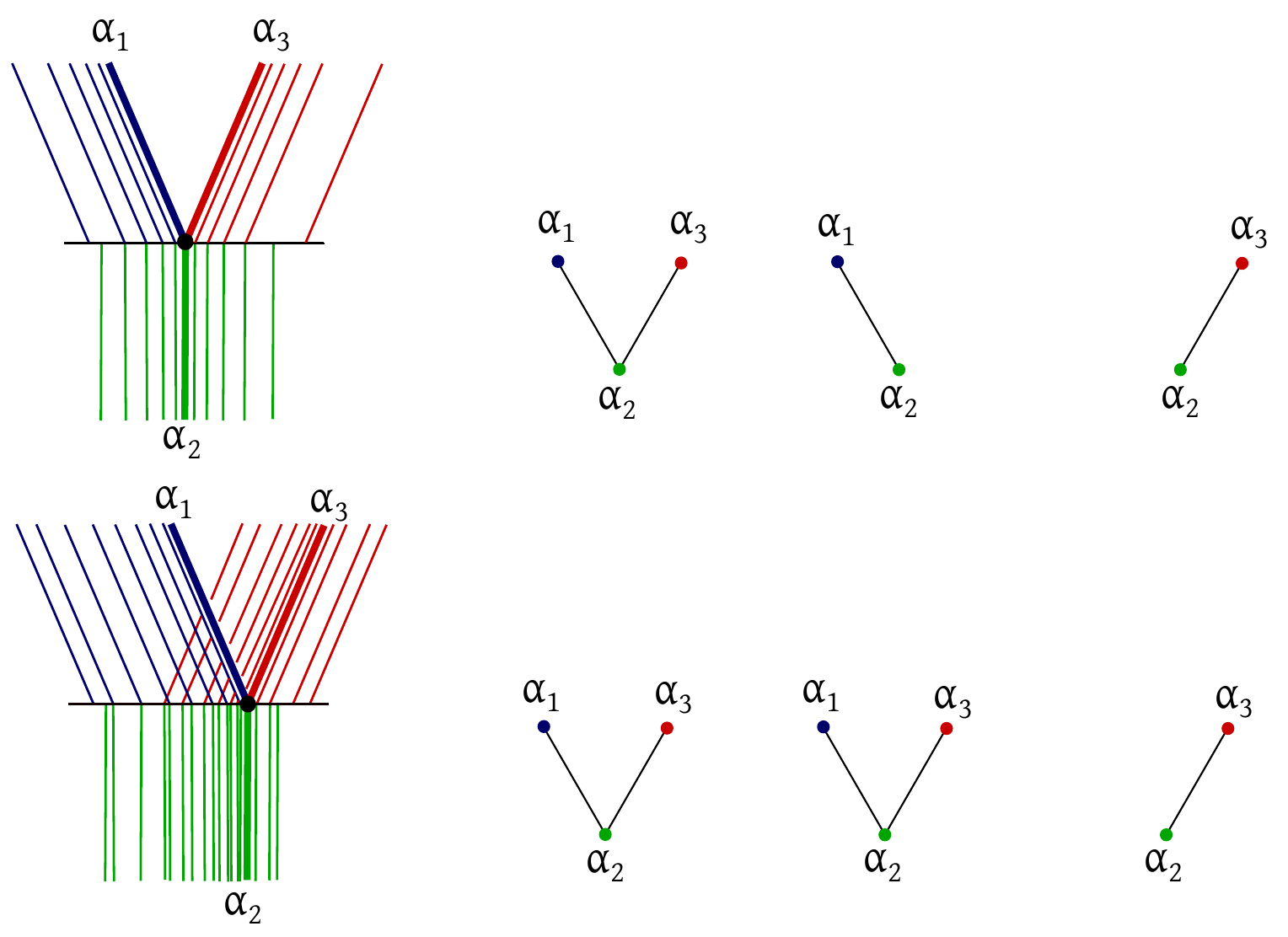}    
     \put(-5,62){$\cB$}
     \put(-5,25){$\cB'$}
     \put(1,54){$d_1$}
     \put(1,18){$d_1$}
     \put(26,54){$d_2$}
     \put(26,18){$d_2$}
     \put(13.5,56.7){$x$}
     \put(22,15){$x$}
     \put(42,60){$\Wh(x)$}
     \put(63,60){$\Wh(x,d_1)$}
     \put(85,60){$\Wh(x,d_2)$}
     \put(42,23){$\Wh(x)$}
     \put(62,23){$\Wh(x,d_1)$}
     \put(85,23){$\Wh(x,d_2)$}     
     \end{overpic}
     \caption{{\small On the left are local pictures of band complexes $\cB$ and $\cB'$; the colored lines represent leaves of the lamination. To the right of each complex is drawn $\Wh(x)$, the {\it ideal Whitehead graph} at $x$, and $\Wh(x,d_i)$, the {\it directional Whitehead graph} at $x$ in direction $d_i$. Note that if the bands overlap, then a directional Whitehead graph may contain more than one edge. 
     }}
     \label{fig_bands_WH}
    \end{figure}

  \begin{remark}
    In the above definition, $\lim_{i \rightarrow \infty} x_i =  x$ since $\cQ$ is continuous with respect to the observers' topology by Theorem~\ref{thm:Q_map}.
  \end{remark}

  \begin{lemma} \label{lemma_WH_same}
    Let $T$ be a very small, indecomposable $F_n$-tree. For every $x \in T$ and direction $d$ at $x$ there exists a compact subtree $K \subset \overline{T}$ and a reduced system of partial isometries $S = (K,A)$ so that $x \in K$, the subtree $K$ contains a germ in the direction $d$, and, 
    \[ \Wh_{T}(x,d) = \Wh_{(K,A)}(x,d). \]
  \end{lemma}
  \begin{proof}
    Let $e$ be an edge of  $\Wh_{T}(x,d)$. By Definition~\ref{def_dirWH}, there exists a leaf $\ell \in L(T)$ corresponding to $e$ so that $\cQ(\ell) = x$, and there exists a sequence $\{\ell_i\}_{i=1}^{\infty} \subset L(T)$ so that $\cQ(\ell_i)=x_i \in d$ and $\lim_{i \rightarrow \infty} \ell_i = \ell$. 
    Let $A$ be a basis of $F_n$ so that $\ell \in C_A$ (a translation of the original basis will do the trick). Since $C_A$ is an open neighborhood of $\ell \in \partial^2 F_n$ there exists an $N \in \N$ so that for all $i>N$, $\ell_i \in C_A$. We truncate the first $N$ elements of the sequence $\{\ell_i\}$. Thus, $x = \cQ_-(\ell) \in \Omega_A$ and $x_i = \cQ_-(\ell_i) \in \Omega_A$. The space $K_A$ is the convex hull of $\Omega_A$ in $\overline{T}$, so $\{x_i\}_{i=1}^\infty \cup \{x\} \subset K_A$. Hence, $K_A$ contains $x$ and the germ corresponding to the direction $d$. By Theorem~\ref{compactHeart}, $L(T) \cap C_A = L(\cB)$. Therefore, $\ell_i \in L(\cB)$ for all $i$, and converges to $\ell \in L(\cB)$. Hence, there is a corresponding edge in $\Wh_{(K,A)}(x,d)$. 
        
    As for the other direction, an edge $e$ in $\Wh_{(K,A)}(x,d)$ corresponds to a leaf $\ell \in L(\cB) = C_A \cap L(T)$ based at $x$ such that there exist leaves $\ell_i \in L(\cB)$ based at $x_i \in d$ so that $\lim_{i \to \infty} \ell_i = \ell$. Thus, $\cQ(\ell_i) = x_i$, $\cQ(\ell)=x$, and  hence $\ell$ corresponds to an edge of $\Wh_{T}(x,d)$. \qedhere

    \end{proof}

\section{The $T_{\pm}$-pattern and $K_{3,3}$ subcomplexes.}\label{sec_k33s}

  \begin{notation}
   Throughout this section, suppose that $\phi \in \out$ is fully irreducible, and let $T_+$ and $T_-$ denote the attracting and repelling trees for $\phi$, respectively.
  \end{notation}
  
An illustration of the following definition appears in Figure \ref{pattern}.

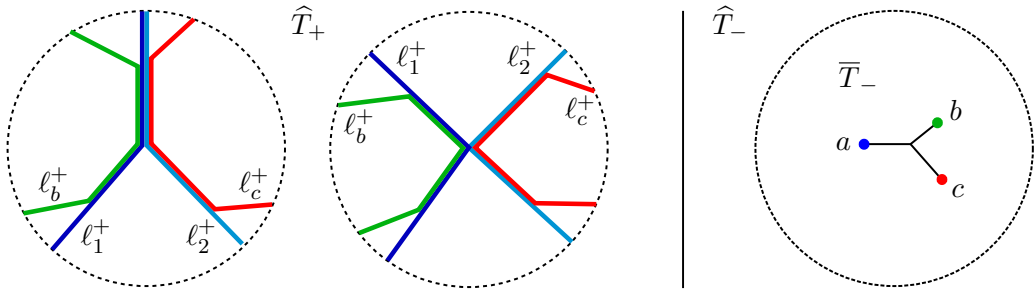
\begin{figure}[b]
     \begin{overpic}[scale=1,tics=5]{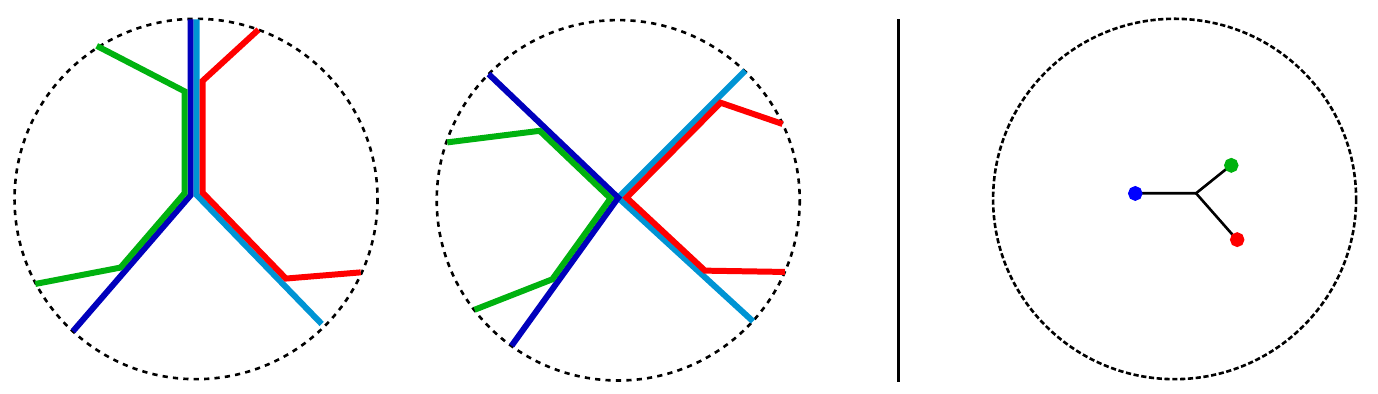}  
     \put(28,25){$\widehat{T}_+$}
     \put(68,25){$\widehat{T}_-$}
     \put(80,20){$\overline{T}_-$}
     \put(79.7,14){$a$}
     \put(90.5,16.8){$b$}
     \put(90.7,9.5){$c$}
     \put(4,10){$\ell_b^+$}
     \put(8,5){$\ell_1^+$}
     \put(18,5){$\ell_2^+$}
     \put(23,10){$\ell_c^+$}
     \put(33,15.5){$\ell_b^+$}
     \put(38,22){$\ell_1^+$}
     \put(48.5,22){$\ell_2^+$}
     \put(54,17){$\ell_c^+$}
     \end{overpic}
     \caption{{\small The $T_\pm$-pattern. The right side of the figure is contained in~$\overline{T}_-$. The configuration in $\widehat{T}_+$ is either like the left side or the middle of the figure depending on whether $(\ell_1^{+} \cap \ell_b^{+}) \cup (\ell_2^{+} \cap \ell_c^{+}) \subset T_+$ in Definition \ref{def_Tpic} has 3 or 4 prongs.}}
     \label{pattern}
    \end{figure}

  \begin{defn} \label{def_Tpic}
    The outer automorphism $\phi$ {\it satisfies the $T_\pm$-pattern} if the following holds. There exists a point $a \in T_-$ such that $|\cQ_-^{-1}(a)| \geq 3$, and there exists a direction $d$ at $a$ containing two points $b,c \in d$ with 
    \[ |\cQ_-^{-1}(b)|=|\cQ_-^{-1}(c)|=2, \]
    which have the following properties. There exist leaves $\ell_1, \ell_2, \ell_b, \ell_c \in L(T_-)$ such that 
    \[
    \begin{array}{l}
    \cQ_-( \ell_1) = \cQ_-( \ell_2) = a \\
    \cQ_-( \ell_b) =b \quad \quad \cQ_-( \ell_c) =c,
    \end{array} 
    \]
    and $(\ell_1^{+} \cap \ell_b^{+}) \cup (\ell_2^{+} \cap \ell_c^{+}) \subset T_+$ is a star with midpoint $y \in \Int(\ell_1^{+} \cap \ell_b^{+}) \cap \Int(\ell_2^{+} \cap \ell_c^{+}) $. 
  \end{defn}

  \begin{remark}
    Suppose $\phi \in \Out(F_n)$ can be represented by an automorphism of $F_n$ which is induced by a pseudo-Anosov homeomorphism of a surface with negative Euler characteristic and non-empty boundary. Then, $\phi$ does not satisfy the $T_{\pm}$-pattern. Indeed, the leaves of the attracting and repelling laminations can be realized as (non-crossing) embedded lines in $\H^2$ with dual trees $\widehat{T}_+$ and $\widehat{T}_-$, respectively. In this case, a direction at a point $a \in T_-$ corresponds to a half-space in $\H^2$. We leave the details to the reader.     
  \end{remark}

\begin{figure}
     \begin{overpic}[scale=1,tics=5]{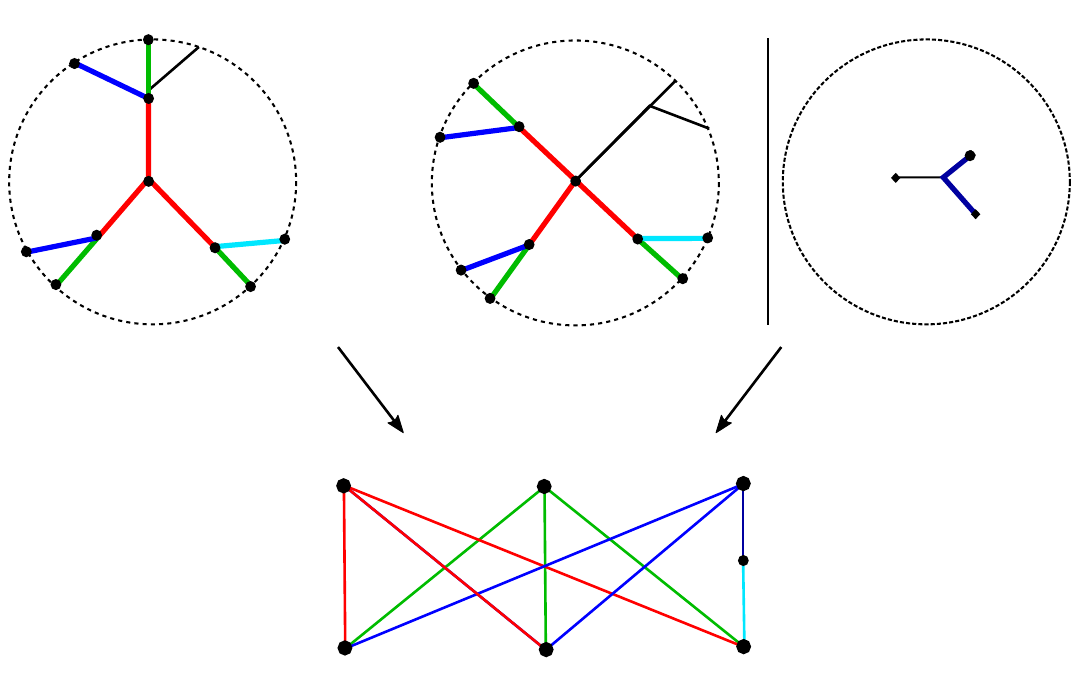}  
	\put(15,10){$\p G_{\phi}$}
	\put(30,60){$\widehat{T}_+$}
	\put(74,60){$\widehat{T}_-$}
	\put(28,27){$\cR_+$}
	\put(72,27){$\cR_-$}
	\put(30,0){$\pi_1$}
	\put(49,0){$\pi_2$}
	\put(68,0){$\pi_3$}
	\put(49,20.5){$\alpha$}
	\put(30,20.5){$\Upsilon$}
	\put(68,20.5){$\beta$}
	\put(70,11){$\gamma$}
	\put(12,61){\Small{$A_1$}}
	\put(2,35){\Small{$A_2$}}
	\put(24,35){\Small{$A_3$}}
	\put(4,59){\Small{$B_1$}}
	\put(-1.8,39){\Small{$B_2$}}
	\put(19,60){\Small{$C_1$}}
	\put(27,40){\Small{$C_2$}}
	\put(15,53){\Small{$p_1$}}
	\put(10,40){\Small{$p_2$}}
	\put(19,43){\Small{$p_3$}}
	\put(11,47){\Small{$y$}}
	\put(39.5,56){\Small{$A_1$}}
	\put(42,33){\Small{$A_2$}}
	\put(63.5,36){\Small{$A_3$}}
	\put(63,57){\Small{$A_4$}}
	\put(36,50){\Small{$B_1$}}
	\put(38,38){\Small{$B_2$}}
	\put(66,51){\Small{$C_1$}}
	\put(66,40){\Small{$C_2$}}
	\put(49,52){\Small{$p_1$}}
	\put(50,39){\Small{$p_2$}}
	\put(59,43){\Small{$p_3$}}
	\put(50,46){\Small{$y$}}
	\put(80,47){\Small{$a$}}
	\put(91,50){\Small{$b$}}
	\put(91,42){\Small{$c$}}	
	\end{overpic}
     \caption{{\small An embedding of $K_{3,3}$ in $\p G_{\phi}$ built using the $T_{\pm}$-pattern. 
     The points $\pi_i$ and $\Upsilon$ are the images of $p_i$ and $y$ under the map $\cR_+$, respectively. In addition, $\alpha = \cR_+(A_i) = \cR_-(a)$, $\beta = \cR_+(B_i) = \cR_-(b)$, and $\gamma = \cR_+(C_i) = \cR_-(c)$. }}
     \label{k33inBoundary}
    \end{figure}
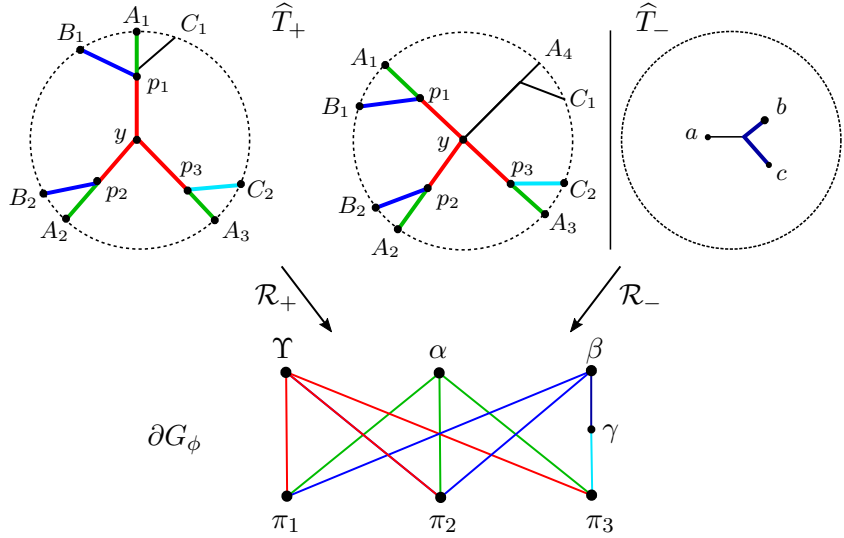
    
  \begin{prop}\label{patternImpliesK33}
    If $\phi$ satisfies the $T_\pm$-pattern (or the $T_\mp$-pattern), then there exists an embedding of $K_{3,3}$ in $\p G_\phi$, where $K_{3,3}$ denotes the complete bipartite graph with two vertex sets of size three.
  \end{prop}
  \begin{proof}
      Suppose $\phi$ satisfies the $T_{\pm}$-pattern. An embedding of $K_{3,3}$ in $\p G_{\phi}$ is illustrated in Figure \ref{k33inBoundary} and is described as follows. Recall from Section \ref{sec:CTmaps} that the Cannon--Thurston map $\hi: \p F_n \rightarrow \p G_{\phi}$ factors through the maps $\cQ_+:\p F_n \rightarrow T_+$ and $\cQ_-:\p F_n \rightarrow \p T_-$. The induced maps $\cR_+:T_+ \rightarrow \p G_{\phi}$ and $\cR_-:T_- \rightarrow \p G_{\phi}$ are embeddings restricted to $\Int(T_+)$ and $\Int(T_-)$, and the images of $\Int(T_+)$ and $\Int(T_-)$ in $\p G_{\phi}$ are disjoint. Thus, the images of the interiors of the colored paths in $\hT_+$ drawn in Figure \ref{pattern} do not intersect in $\p G_{\phi}$. The endpoints of these paths on $\p T_+$ are identified, via $\cR_+$ and $\cR_-$, to points in the interior of $T_-$ to form a $K_{3,3}$.
  \end{proof}

  \begin{lemma}\label{paternIffDw}
    An outer automorphism $\phi$ satisfies the $T_{\pm}$-pattern if and only if there exists a point $a \in T_-$ and a direction $d$ at $a$ such that $\Wh_{T_-}(a,d)$ contains at least three vertices.
  \end{lemma}

  \begin{proof}
    Suppose there exist a point $a \in T_-$ and a direction $d$ at $a$ so that $\Wh_{T_-}(a,d)$ contains three vertices, corresponding to points $\ga_1, \ga_2, \ga_3 \in \partial F_n$. By definition, $\cQ_-(\ga_i)  = a$ and each $\ga_i$ is an end of a leaf of the lamination $L(T_-)$. Thus, there exist two leaves $\ell_1 \neq \ell_2 \in L(T_-)$ so that $\{\ga_1, \ga_2, \ga_3\} \subset (\ell_1 \cup \ell_2)$. By the definition of the directional Whitehead graph, there exist sequences $\{\sig_k\}_{k=1}^{\infty},\{\rho_k\}_{k=1}^{\infty} \subset L(T_-)$ such that $\lim_{k\to \infty} \sig_k = \ell_1$, $\lim_{k\to \infty} \rho_k = \ell_2$, and $\cQ_-(\sig_k), \cQ_-(\rho_k) \in d$ for all $k$. By Proposition~\ref{fiberIsStar}, the realization in $T_+$ of the leaves corresponding to edges of a Whitehead graph is a star. Therefore, $\ell_1^+ \cap \ell_2^+$ is either a point~$y$ or a ray initiating at a point $y$. Since $\lim_{k\to \infty} \sig_k = \ell_1$, by Proposition~\ref{PropOfRelazationsOfLam} $\sig_k^+ \cap \ell_1^+$ is a non-trivial segment containing $y$ in its interior for a large enough $k$. Likewise, for large enough $k$, $\rho_k^+\cap \ell_2^+$ is a non-trivial segment containing $y$ in its interior. Set $\ell_a=\sig_k$ and $\ell_b = \rho_k$  for large enough $k$ so that $\bigl( \ell_1^{+} \cap \ell_b^{+} \bigr) \cup \bigl( \ell_2^{+} \cap \ell_c^{+} \bigr)$ is a star with $y$ its midpoint. Therefore, $\phi$ satisfies the $T_{\pm}$-pattern, concluding one direction of the proof.

    Conversely, suppose $\phi$ satisfies the $T_{\pm}$-pattern. We will show $\ell_1$ and $\ell_2$ yield distinct edges of $Wh_{T_-}(a,d)$. By Lemma \ref{lem_prin_aut}, there exists a homothety $h_+:T_+ \rightarrow T_+$ that fixes the middle $y$ of the star $(\ell_1^{+} \cap \ell_b^{+}) \cup (\ell_2^{+} \cap \ell_c^{+})$ in $T_+$. The homothety $h_+$ corresponds to a principal automorphism $\Phi$ representing $\phi$. We replace $\phi$ by a rotationless power so that $h_+$ fixes all directions at $y$ and  $\partial\Phi(\ell_j) = \ell_j$ for $j=1,2$. Let $h_-$ be the corresponding homothety of $T_-$ guaranteed by Lemma \ref{lem_prin_aut}. By  Lemma \ref{Qandh},
    \[ h_-(a) = h_-( \cQ_-( \ell_1)) = \cQ_-( \Phi (\ell_1)) = \cQ_-(\ell_1) = a, \] where $a \in T_-$ is the point given in the definition of the $T_{\pm}$-pattern. Thus, $h_-$ permutes the directions at $a$. We replace $\Phi, h_+$, and $h_-$ by powers so that $h_-$ fixes the directions at $a$. Since $h_+$ is a homothety fixing the directions at $y$, $\lim_{k \to \infty} h_+^k(\ell_b^+) = \ell_1^+$. By Lemma \ref{PropOfRelazationsOfLam},  $\lim_{k \to \infty} \Phi^k(\ell_b) = \ell_1$. Likewise, $\lim_{k \to \infty} \Phi^k(\ell_c) = \ell_2$. Recall that $b = \cQ_-(\ell_b)$ and $c = \cQ_-(\ell_c)$ belong to the same component $d$ of $T_- \setminus \{a\}$ by assumption. Since $h_-$ fixes the directions at $a$, 
    \[ \cQ_-(\Phi^k(\ell_b)) = h_-^k(\cQ_-(\ell_b)) \in d \] 
    and likewise, $\cQ_-(\Phi^k(\ell_c)) \in d$. Therefore, $\ell_1$ and $\ell_2$ yield distinct edges in $\Wh_{T_-}(a,d)$, concluding the proof.
  \end{proof}

  \begin{lemma}\label{DWvsTLevitt}
     There exists a directional Whitehead graph of a point in $T_-$ with at least three vertices if and only if the Rips Machine for $T_-$ never halts.
  \end{lemma}
  \begin{proof}
    Suppose the Rips machine never halts, and let $(K_m, A_m)$ denote the output after the $m^{th}$ iteration of the Rips machine. 
    This means that the tree $T_-$ is not of surface type, and thus, Proposition~\ref{prop:Levitt} ensures that $T_-$ is of Levitt type. In particular, if $a_m \in A_m$ (for each $m\in\N$), then $\lim_{m \rightarrow \infty}\diam(\dom(a_{m}))=0$. 
   
    By Proposition \ref{prop:levitt implies overlap}, for each $m$ there exists a compact non-trivial interval $I_m \subset K_m$ and three distinct elements $a_m, b_m, c_m \in A_m$ such that 
    $\displaystyle I_m  \subset  \dom(a_m) \cap dom(b_m) \cap dom(c_m)$. We may choose $\{I_m\}$ to be nested because for each $m$, if an interval is contained in the domain of three partial isometries in $A_m$, then it is contained in the domain of three partial isometries in $A_{m-1}$. We may also choose the sequences of partial isometries so that for each $f=a,b,c$, $f_{m}|_{K_{m+1}}=f_{m+1}$. Thus, there exists a point~$x$ with $\displaystyle x \in \bigcap_{m=1}^\infty I_m$. Moreover, because there are only finitely many directions at $x$, there exists at least one germ $d$ based at $x$, so that for each $m$, the interval $I_m$ contains a subsegment in the germ $d$.

    To find an infinite sequence of leaves, let $x_m \neq x$ be an extremal point of  $\dom(a_{m+1})$  so that the segment $[x,x_m]$ is contained in the germ $d$. By Proposition~\ref{prop_reduced}, $x_m \in \Omega$, where $\Omega$ is the limit set of the system (see Definition~\ref{def_systems}). So, there exists a leaf  $\sig_m \in L(\cB_-)$ with $\cQ_-(\sig_m) = x_m$. In the alphabet $A_0$, $\sig_m = \cdots.a_0 \cdots$ for all $m \in \N$; see Notation~\ref{notation_leaf_dot}. Since $L(\cB_-)$ is compact, there is a subsequence of $\{\sig_m\}_{m \in \N}$ that converges in $\partial^2 F_n$ to a leaf $\ell_1$. Moreover, $\diam (\dom(a_m)) \to 0$, so $\lim_{m \rightarrow \infty} x_m = x$. The map $\cQ_-$ is continuous with respect to the observers' topology, so $\cQ_-(\ell_1) = \lim_{m \rightarrow \infty} \cQ_-(\sig_m) = \lim_{m \rightarrow \infty} x_m = x$. Thus, there exists an edge $e$ in $\Wh_K(x,d)$ corresponding to $\ell_1$. Since $x_m \in \dom(a_m)$, in the alphabet $A_m$, $\sig_m = \cdots.a_m \cdots$. Thus, one half-leaf of $\ell_1$ is the limit of half-leaves which begin with $\{a_m\}_{m\in \N}$. Without loss of generality, assume that the other end of $\ell_1$ is the limit of half-leaves which begin with $\{b_m\}_{m \in \N}$. Perform the same construction with $\{c_m\}_{m \in \N}$ instead of $\{a_m\}_{m \in \N}$ to obtain a sequence $\{y_m\}_{m \in \N} \subset d$ and a sequence $\{\rho_m\}_{m \in \N} \subset \cL(\cB_-)$ so that  $\rho_m = \cdots. c_m \cdots$ in the alphabet $A_m$ and $\lim_{m \rightarrow \infty} \rho_m = \ell_2$. Then $\cQ_-(\ell_2) = x$. The leaf $\ell_2$ is the limit of half-leaves which begin with $\{c_m\}_{m \in \N}$, so $\ell_1 \neq \ell_2$. Therefore, there exists another edge $e' \neq e$ in $\Wh_{(K,A)}(x,d)$ where $e'$ corresponds to $\ell_2$, and $e$ corresponds to $\ell_1$. By Lemma~\ref{lemma_WH_same}, $\Wh_{T_-}(x,d)$ contains two edges as well. 

    Conversely, suppose $\Wh_{T_-}(x,d)$ contains three vertices. By Lemma \ref{lemma_WH_same}, there exists a system of partial isometries $(K,A)$ so that $\Wh_{(K,A)}(x,d)$ contains three vertices. So, $\Wh_{(K,A)}(x,d)$ contains two edges corresponding to leaves $\ell_1, \ell_2$ in $L(\cB_-)$. By definition, there exist sequences $\{ \sig_k \}_{k \in \N}, \{\rho_k\}_{k \in \N} \subset L(\cB_-)$ so that $\lim_{k \rightarrow \infty} \sig_k = \ell_1, \lim_{k \rightarrow \infty} \rho_k = \ell_2$, and $\cQ_-(\sig_k), \cQ_-(\rho_k) \in d$. For each $m \in \N$, there are two distinct pairs $(a,b), (c,f)$ of partial isometries $a,b,c,f \in A_m$ so that in the alphabet of $A_m$, $\ell_1 = \cdots b^{-1}.a \cdots$ and $\ell_2 = \cdots f^{-1}.c \cdots$. (Note that the pairs may share one element.) Therefore,  $x \in \dom(a) \cap dom(b) \cap dom(c) \cap dom(f)$. Let $k$ be large enough so that in the alphabet $A_m$, $\sig_k = \cdots b^{-1}.a \cdots$ and $\rho_k = \cdots f^{-1}.c \cdots$. Thus, $\cQ_-(\sig_k) \in dom(a) \cap dom(b)$ and $\cQ(\rho_k) \in dom(c) \cap dom(f)$. Since the domain of a partial isometry is convex, the arc $[\cQ_-(\sig_k),x]$ is contained in $dom(a) \cap dom(b)$, and the arc $[\cQ_-(\rho_k),x]$ is contained in $dom(c) \cap dom(f)$. Hence, the domains of $a,b,c,f$ intersect in a non-trivial arc (in the direction $d$). By \cite{coulboishilion}, there exists a free band, that is, an interval $J \in K$ and a unique partial isometry $g \in A_m$ so that $J \subset dom(g)$. Thus, for all $m \in \N$ the Rips Machine does not stop at the $m$-th step.
\end{proof}

\section{The boundary $\partial G_\phi$} \label{sec_boundary}

  \begin{Mainthm}
    If $G_\phi$ is a hyperbolic group, then $\partial G_\phi$ contains a copy of the complete bipartite graph $K_{3,3}$. 
  \end{Mainthm}

  \begin{proof}
    If $\phi$ is fully irreducible, let $T_+$ and $T_-$ be the attracting and repelling trees of $\phi$ in the boundary of $CV_n$. By Proposition~\ref{prop:Levitt}, either $T_-$ or $T_+$ has Levitt type. 
    Without loss of generality, suppose $T_-$ has Levitt type.
    By Lemma \ref{DWvsTLevitt}, some directional Whitehead graph of $T_-$ contains more than one edge. 
    By Lemma \ref{paternIffDw}, $\phi$ satisfies the $T_\pm$ pattern; hence, by Lemma \ref{patternImpliesK33}, the boundary $\partial G_\phi$ contains an embedded copy of $K_{3,3}$.  

    If $\phi$ is not fully irreducible, then there exists a free 	factor $A_0<F_n$ and a power $\phi^k$ so that $[A_0]$, the conjugacy class of $A_0$ in $F_n$, is $\phi^k$-invariant. Let $A$ be a minimal such factor. Since $G_\phi$ is hyperbolic, $\phi$ is atoroidal, hence $rank(A) \geq 2$. Moreover, $\phi$ induces a well-defined outer automorphism $\phi'$ on $A$ \cite[Fact 1.4]{hm13}. Since $\phi$ is atoroidal, so is $\phi'$, and hence $rank(A) \geq 3$. By the minimality of $A$, the outer automorphism $\phi'$ is fully irreducible. Therefore, the boundary of $G' = A \rtimes_{\phi'} \Z$ contains an embedded copy of $K_{3,3}$ by the previous paragraph. The subgroup $G'$ quasi-isometrically embeds in $G_{\phi^k}$ \cite{mitra}. Thus, $\p G'$ embeds into $\p G_{\phi^k}$. The group $G_{\phi^k}$ is a finite-index subgroup of $G_\phi$, so the boundaries of $G_{\phi}$ and $G_{\phi^k}$ are homeomorphic. 
  \end{proof}

  In what follows, we describe the work of Kapovich-Kleiner \cite{kapovichkleiner} relevant to this paper.

  \begin{thm}\cite[Theorem 4]{kapovichkleiner}\label{KKthm}
    Let $G$ be a hyperbolic group which does not split over a finite or virtually cyclic subgroup, and suppose $\partial G$ is 1-dimensional. Then one of the following holds:
  \begin{enumerate}	
    \item $\partial G$ is a Menger curve,
    \item $\partial G$ is a Sierpinski carpet,
    \item $\partial G$ is homeomorphic to $\mathbb{S}^1$, and $G$ maps onto a Schwartz triangle group with finite kernel. 
      \qed
  \end{enumerate}
  \end{thm}

  Theorem \ref{KKthm} quickly follows from a compilation of results: the characterization of the Menger curve \cite{anderson58a,anderson58b},  the characterization of the Sierpinski carpet \cite{whyburn}, and the structure of local and global cut points in the boundary of a hyperbolic group \cite{bestvinamess,bowditch99,swarup}.

  \begin{cor}\cite[Corollary 15]{kapovichkleiner}\label{KK_FbyZwMengerC_boundary}
    Let $F$ be a finitely generated free group, and $\phi \from F \to F$ an atoroidal automorphism. If no power of $\phi$ preserves a free splitting of $F$, then the Gromov boundary of $G_\phi := F \rtimes_\phi \Z$ is the Menger curve. In particular, if $\phi$ is fully irreducible, then $\partial G_{\phi}$ is the Menger curve. 
    \qed
  \end{cor}
  
  \begin{remark} 
    The hypothesis of the above corollary in \cite{kapovichkleiner} states that $\phi$ is irreducible, not fully irreducible. However, in the proof they obtain a contradiction by producing a $\phi^j$-invariant free decomposition of $F$ for some positive $j>0$. This property only contradicts that $\phi$ is fully irreducible.
  \end{remark}

  To prove the above corollary, Kapovich-Kleiner \cite{kapovichkleiner} first show that if $\phi \in \out$ is fully irreducible (or no power of $\phi$ preserves a free splitting of $F$), then $G_\phi$ does not split over a trivial or cyclic subgroup. (See also Brinkmann \cite{brinkmann02}.) 
  Then, Theorem~\ref{KKthm} implies $\partial G_\phi$ is either the circle, the Sierpinski carpet, or the Menger curve. At this point, using our Main Theorem, one can directly rule out the circle and Sierpinski carpet, since these spaces are planar. Alternatively, Kapovich--Kleiner rule out the circle by applying the work of Tukia--Gabai--Casson--Jungreis \cite{tukia,gabai,cassonjungreis}, which classifies the hyperbolic groups with boundary homeomorphic to $S^1$ as precisely the groups that act discretely and cocompactly by isometries on the hyperbolic plane. Finally, Kapovich--Kleiner prove a rather deep result: if $G$ is a hyperbolic group with Sierpinski carpet boundary, then $G$ together with the stabilizers of the peripheral circles of the \scarpet carpet forms a {\it Poincar\'{e} duality pair} (see \cite[Corollary 12]{kapovichkleiner}). In particular, they conclude that in this case $\chi(G)<0$. Since $\chi(G_\phi)=0$, the boundary $\partial G_\phi$ is the Menger curve.

\bibliographystyle{alpha}
\bibliography{Ref}

\begin{thebibliography}{GJLL98}

\bibitem[And58a]{anderson58a}
R.~D. Anderson.
\newblock A characterization of the universal curve and a proof of its
  homogeneity.
\newblock {\em Ann. of Math. (2)}, 67:313--324, 1958.

\bibitem[And58b]{anderson58b}
R.~D. Anderson.
\newblock One-dimensional continuous curves and a homogeneity theorem.
\newblock {\em Ann. of Math. (2)}, 68:1--16, 1958.

\bibitem[BF94]{OuterLimits}
Mladen Bestvina and Mark Feighn.
\newblock Outer limts, 1994.
\newblock Preprint available at
  {http://andromeda.rutgers.edu/~feighn/research.html}.

\bibitem[BFH97]{tits0}
M.~Bestvina, M.~Feighn, and M.~Handel.
\newblock Laminations, trees, and irreducible automorphisms of free groups.
\newblock {\em Geom. Funct. Anal.}, 7(2):215--244, 1997.

\bibitem[BM91]{bestvinamess}
Mladen Bestvina and Geoffrey Mess.
\newblock The boundary of negatively curved groups.
\newblock {\em J. Amer. Math. Soc.}, 4(3):469--481, 1991.

\bibitem[Bow98]{bowditch}
Brian~H. Bowditch.
\newblock Cut points and canonical splittings of hyperbolic groups.
\newblock {\em Acta Math.}, 180(2):145--186, 1998.

\bibitem[Bow99]{bowditch99}
B.~H. Bowditch.
\newblock Connectedness properties of limit sets.
\newblock {\em Trans. Amer. Math. Soc.}, 351(9):3673--3686, 1999.

\bibitem[Bri00]{brinkmann00}
P.~Brinkmann.
\newblock Hyperbolic automorphisms of free groups.
\newblock {\em Geom. Funct. Anal.}, 10(5):1071--1089, 2000.

\bibitem[Bri02]{brinkmann02}
Peter Brinkmann.
\newblock Splittings of mapping tori of free group automorphisms.
\newblock {\em Geom. Dedicata}, 93:191--203, 2002.

\bibitem[CH12]{coulboishilion_bot}
Thierry Coulbois and Arnaud Hilion.
\newblock Botany of irreducible automorphisms of free groups.
\newblock {\em Pacific J. Math.}, 256(2):291--307, 2012.

\bibitem[CH14]{coulboishilion}
Thierry Coulbois and Arnaud Hilion.
\newblock Rips induction: index of the dual lamination of an {$\Bbb{R}$}-tree.
\newblock {\em Groups Geom. Dyn.}, 8(1):97--134, 2014.

\bibitem[CHL07]{coulboishilionlustig_ot}
Thierry Coulbois, Arnaud Hilion, and Martin Lustig.
\newblock Non-unique ergodicity, observers' topology and the dual algebraic
  lamination for {$\Bbb R$}-trees.
\newblock {\em Illinois J. Math.}, 51(3):897--911, 2007.

\bibitem[CHL09]{coulboishilionlustig}
Thierry Coulbois, Arnaud Hilion, and Martin Lustig.
\newblock {$\Bbb R$}-trees, dual laminations and compact systems of partial
  isometries.
\newblock {\em Math. Proc. Cambridge Philos. Soc.}, 147(2):345--368, 2009.

\bibitem[CHR15]{chr}
Thierry Coulbois, Arnaud Hilion, and Patrick Reynolds.
\newblock Indecomposable {$F_N$}-trees and minimal laminations.
\newblock {\em Groups Geom. Dyn.}, 9(2):567--597, 2015.

\bibitem[CJ94]{cassonjungreis}
Andrew Casson and Douglas Jungreis.
\newblock Convergence groups and {S}eifert fibered {$3$}-manifolds.
\newblock {\em Invent. Math.}, 118(3):441--456, 1994.

\bibitem[CL95]{CohenLustig}
Marshall~M. Cohen and Martin Lustig.
\newblock Very small group actions on {${\bf R}$}-trees and {D}ehn twist
  automorphisms.
\newblock {\em Topology}, 34(3):575--617, 1995.

\bibitem[CM87]{cullermorgan}
Marc Culler and John~W. Morgan.
\newblock Group actions on {${\bf R}$}-trees.
\newblock {\em Proc. London Math. Soc. (3)}, 55(3):571--604, 1987.

\bibitem[CV86]{cullervogtmann}
Marc Culler and Karen Vogtmann.
\newblock Moduli of graphs and automorphisms of free groups.
\newblock {\em Invent. Math.}, 84(1):91--119, 1986.

\bibitem[FH09]{fh09}
M.~Feighn and M.~Handel.
\newblock Abelian subgroups of {${\rm Out}(F_n)$}.
\newblock {\em Geom. Topol.}, 13(3):1657--1727, 2009.

\bibitem[Gab92]{gabai}
David Gabai.
\newblock Convergence groups are {F}uchsian groups.
\newblock {\em Ann. of Math. (2)}, 136(3):447--510, 1992.

\bibitem[GJLL98]{GJLL}
Damien Gaboriau, Andre Jaeger, Gilbert Levitt, and Martin Lustig.
\newblock An index for counting fixed points of automorphisms of free groups.
\newblock {\em Duke Math. J.}, 93(3):425--452, 1998.

\bibitem[Hil07]{hilion_stab}
Arnaud Hilion.
\newblock On the maximal subgroup of automorphisms of a free group {$F_N$}
  which fix a point of the boundary {$\partial F_N$}.
\newblock {\em Int. Math. Res. Not. IMRN}, (18):Art. ID rnm066, 21, 2007.

\bibitem[HK09]{hruskakleiner}
G.~Christopher Hruska and Bruce Kleiner.
\newblock Erratum to: ``{H}adamard spaces with isolated flats'' [{G}eom.
  {T}opol. {\bf 9} (2005), 1501--1538; mr2175151].
\newblock {\em Geom. Topol.}, 13(2):699--707, 2009.

\bibitem[HM11]{handelMosher_axes}
M.~Handel and L.~Mosher.
\newblock {\em Axes in {O}uter {S}pace}.
\newblock Number 1004. Amer Mathematical Society, 2011.

\bibitem[HM13]{hm13}
Michael Handel and Lee Mosher.
\newblock Subgroup decomposition in {O}ut({F}$_n$), part {I}: {G}eometric
  {M}odels.
\newblock {\em arXiv:1302.2378}, 2013.

\bibitem[KK00]{kapovichkleiner}
Michael Kapovich and Bruce Kleiner.
\newblock Hyperbolic groups with low-dimensional boundary.
\newblock {\em Ann. Sci. \'Ecole Norm. Sup. (4)}, 33(5):647--669, 2000.

\bibitem[KL14]{kapovichlustig_laminations}
Ilya Kapovich and Martin Lustig.
\newblock Invariant laminations for irreducible automorphisms of free groups.
\newblock {\em Q. J. Math.}, 65(4):1241--1275, 2014.

\bibitem[KL15]{kapovichlustig}
Ilya Kapovich and Martin Lustig.
\newblock Cannon-{T}hurston fibers for iwip automorphisms of {$F\sb N$}.
\newblock {\em J. Lond. Math. Soc. (2)}, 91(1):203--224, 2015.

\bibitem[Kur30]{kuratowski}
Kazimierz Kuratowski.
\newblock Sur le problème des courbes gauches en topologie.
\newblock {\em Fund. Math.}, 15(3):271--283, 1930.

\bibitem[LL03]{levittlustig}
Gilbert Levitt and Martin Lustig.
\newblock Irreducible automorphisms of {$F\sb n$} have north-south dynamics on
  compactified outer space.
\newblock {\em J. Inst. Math. Jussieu}, 2(1):59--72, 2003.

\bibitem[Mit97]{mitralaminations}
Mahan Mitra.
\newblock Ending laminations for hyperbolic group extensions.
\newblock {\em GAFA}, pages 379--402, 1997.

\bibitem[Mit98]{mitra}
Mahan Mitra.
\newblock Cannon-{T}hurston maps for hyperbolic group extensions.
\newblock {\em Topology}, 37(3):527--538, 1998.

\bibitem[Pau96]{p96}
Fr\'ed\'eric Paulin.
\newblock Un groupe hyperbolique est d\'etermin\'e par son bord.
\newblock {\em J. London Math. Soc. (2)}, 54(1):50--74, 1996.

\bibitem[Rua05]{ruane}
Kim Ruane.
\newblock C{AT}(0) boundaries of truncated hyperbolic space.
\newblock {\em Topology Proc.}, 29(1):317--331, 2005.
\newblock Spring Topology and Dynamical Systems Conference.

\bibitem[Swa96]{swarup}
G.~A. Swarup.
\newblock On the cut point conjecture.
\newblock {\em Electron. Res. Announc. Amer. Math. Soc.}, 2(2):98--100, 1996.

\bibitem[Tuk88]{tukia}
Pekka Tukia.
\newblock Homeomorphic conjugates of {F}uchsian groups.
\newblock {\em J. Reine Angew. Math.}, 391:1--54, 1988.

\bibitem[Vog02]{vogtmann}
Karen Vogtmann.
\newblock Automorphisms of free groups and outer space.
\newblock In {\em Proceedings of the {C}onference on {G}eometric and
  {C}ombinatorial {G}roup {T}heory, {P}art {I} ({H}aifa, 2000)}, volume~94,
  pages 1--31, 2002.

\bibitem[Why58]{whyburn}
G.~T. Whyburn.
\newblock Topological characterization of the {S}ierpi\'nski curve.
\newblock {\em Fund. Math.}, 45:320--324, 1958.

\end{thebibliography}

\end{document}